\documentclass[journal,twoside,web]{ieeecolor}
\usepackage{generic}

\usepackage{graphicx}      

\usepackage{amsmath}
\usepackage{amssymb}
\usepackage{bbm}
\usepackage{subcaption}
\usepackage{dsfont}
\usepackage{algpseudocode}
\usepackage{macros}

\usepackage[utf8]{inputenc}
\usepackage[english]{babel}

\newcommand{\subparagraph}{}
\usepackage{titlesec}

\titlespacing{\section}{0pt}{4pt plus 0pt minus 1pt}{0pt plus 0pt minus 1pt}
\titlespacing{\subsection}{0pt}{4pt plus 0pt minus 1pt}{0pt plus 0pt minus 1pt}
\usepackage[font=small,skip=2pt]{caption}
\newtheorem{theorem}{Theorem}
\newtheorem{lemma}{Lemma}

\newtheorem{assumption}{Assumption}
\newtheorem{remark}{Remark}

\begin{document}
	
	\title{On Convergence Analysis of Network-GIANT: An approximate Hessian-based fully distributed optimization algorithm}
    \author{Souvik Das, Luca Schenato, \IEEEmembership{Fellow,IEEE}, and Subhrakanti Dey, \IEEEmembership{Fellow, IEEE}
\thanks{This publication has emanated from research conducted with the financial support of Swedish Research Council (VR) Grant 2023-04232.}
\thanks{Souvik Das and Subhrakanti Dey are with the Department of Electrical Engineering, Uppsala University, Sweden.
(e-mail: \{souvik.das, subhrakanti.dey\}@angstrom.uu.se).}
\thanks{Luca Schenato is with the Department of Information Engineering, University of Padova, Italy. 
e-mail: (schenato@dei.unipd.it).}
}


\maketitle
\thispagestyle{empty}

	
	\begin{abstract}
        This paper presents a detailed convergence and performance analysis of a recently developed approximate Newton-type fully distributed optimization method for \(L\)-smooth, \(\mu\)-strongly convex local loss functions, called Network-GIANT (inspired by the Federated learning algorithm GIANT possessing mixed linear-quadratic convergence properties). Network-GIANT has been empirically seen to achieve faster linear convergence properties compared to its gradient-based counterparts, and several other existing second order distributed algorithms, while having the same communication complexity (per iteration) as its first order distributed counterparts. We first explicitly characterize a \emph{global linear convergence rate} for Network-GIANT, which can be computed as the spectral radius of a $3 \times 3$ matrix dependent on $L$, $\mu$, and the spectral norm ($\sigma$) of the consensus matrix of the underlying undirected graph. We provide an explicit bound on the step size parameter 
        $\eta$, below which this spectral radius is guaranteed to be less than $1$. Furthermore, we derive a mixed 
        linear-quadratic inequality based upper bound for the optimality gap norm, and provide a rigorous proof of a local asymptotic convergence 
        rate of \(1 - \eta \big(1 - \frac{\gamma}{\mu}\big)\) given the Hessian approximation error $\gamma < \mu$, which formally explains the faster convergence rate of Network-GIANT. Numerical experiments are carried out with a reduced CovType dataset for binary logistic regression over a variety of graphs, including heterogeneous data distributions, to illustrate the above theoretical results.
	\end{abstract}
	
\section{Introduction}
\label{sec:intro}
Distributed optimization and learning generally refer to the paradigm where multiple agents optimize a global cost function or train a global model collaboratively, often by optimizing their local cost functions or training a local model based on their local data sets. In applications such as networked autonomous systems, Internet of Things (IoT), collaborative robotics, and smart manufacturing, moving large amounts of locally collected data to a central processing unit can be expensive from both a communication and storage point of view. Additionally, distributed optimization/learning prevents the sharing of raw data, thus guaranteeing data privacy to a certain extent. Distributed/decentralized optimization or learning can occur over a network consisting of a central server with multiple local nodes (Federated Learning), or over a fully distributed setting without a central server, where nodes only communicate with their neighbours. There has been significant progress in recent years in both settings - see \cite{ref:LY-ZW-LS-PSY-CGB-24,ref:CL-NB-WH-YS-KHJ-24, yang_dist_opt_survey} for an extensive survey on these topics focusing on their key challenges and opportunities. 

In this work, we will primarily focus on the topic of fully distributed optimization, where nodes collaboratively optimize a global cost, which is a sum of local cost functions. We will also restrict ourselves to strongly convex cost functions.
In this context, distributed optimization or learning algorithms proceed in an iterative fashion, where each node updates its local estimate of the global optimization variable based on consensus-type averaging of its own information and the information received from its neighbours, combined with (typically) a gradient-descent type algorithm with a suitably chosen step size (constant or diminishing with time).  Without going into a detailed literature survey of gradient based distributed optimization algorithms, we refer the readers to \cite{nedic_annual_review, dist_grad_ML}. For strongly convex functions, it is well known that gradient based techniques achieve linear (exponential) convergence with a constant step size. Replicating this in a fully distributed setting requires a gradient-tracking algorithm where each node keeps a vector that keeps a local estimate of the global gradient \cite{li_na_gradient_tracking}. A detailed convergence analysis was carried out in \cite{li_na_gradient_tracking} to illustrate linear convergence of gradient-tracking based distributed first order optimization algorithms for smooth and strongly convex functions over undirected graphs, whereas a sublinear convergence rate $O(1/k)$ was guaranteed for smooth convex functions that are not strongly convex. A gradient-tracking based distributed first-order optimization with an additional distributed eigenvector tracking was presented in \cite{usman_directed}, with a provable linear convergence rate.

It is well known that for strongly convex and smooth functions, centralized optimization algorithms using curvature (or Hessian) information of the cost function, such as the Newton-Raphson method, achieve faster locally quadratic convergence in the vicinity of the optimum. In the context of fully distributed optimization algorithms utilizing Newton-type methods, achieving quadratic convergence is difficult, primarily due to the fact that communicating local Hessian information to neighbouring nodes is infeasible for large model size or parameter dimension, and the underlying consensus algorithm essentially slows down the convergence to at best a linear rate. Computation of the Hessian and its inversion is another bottleneck, which has motivated a plethora of recent approximate Newton-type {\em Federated learning} algorithms \cite{ref:CTD-NHT-TDN-WB-ARB-BBZ-AYZ-22,ref:GS-SD-22, ref:MS-RI-XQ-PR-22, ref:MS-AKQ-KMC-24}, which can also display a mixed linear-quadratic convergence rate. In the fully distributed regime, earlier examples of second order optimization algorithms  include \cite{ref:AM-QL-AR-16,ref:DB-DJ-NK-NJ-17,ref:NB-RC-GN-LC-DV-18, ref:JZ-QL-MCA-21} that approximate the Hessian based on Taylor series expansion and historical data. More recent works, such as  \cite{li2020communication},  developed a fully distributed version of DANE \cite{ref:OS-NS-TZ-14} (termed as Network-DANE) to overcome the high computational costs of the existing algorithms, and \cite{ref:ZQ-XL-LL-YH-23} captured the second order information of the objective functions based on an augmented Lagrangian function along with gradient diffusion technique to come up with a distributed algorithm leveraging the Hessian information of the objective function. Recent papers such as  \cite{maritan2024fully} and \cite{zhang_you_adaptive_newton} have shown, however, that the superlinear convergence properties of Federated second order optimization algorithms such as GIANT can be retrieved only when exact averages of gradients or parameters can  be computed via finite-time exact consensus 
\cite{maritan2024fully}, or through distributed finite-time set consensus \cite{zhang_you_adaptive_newton}, both of which can require up to $O(N)$ (where $N$ is the number of nodes) consensus rounds between two successive optimization iterations, thus making them infeasible for large networks.

In this article, we consider a recently developed {\em fully distributed approximate Newton-type optimization algorithm} called \(\ngiant\) \cite{alessio_net_giant}, which combined gradient tracking with consensus based updates and a descent direction computed by the product of the inverse local Hessian and the gradient tracking vector at each node. Built on an extension of the approximate Newton-type federated learning algorithm \(\giant\)  \cite{wang2018giant}, it was shown that \(\ngiant\)  achieves semi-global exponential convergence to the exact optimal solution for a sufficiently small step-size, and enjoys a low communication overhead of $O(n)$  per iteration at each node, where \(n\) is the dimension of the decision space, thus making it comparable to first-order distributed optimization techniques.
 The authors of \cite{alessio_net_giant} empirically also demonstrated a superior convergence rate compared to gradient-tracking based first order methods, as well as several other second order distributed optimization algorithms, including the well-known Network-DANE \cite{ref:HY-CH-XC}.  Although \cite{alessio_net_giant} proved semi-global exponential convergence using a two-time-scale separation principle that guarantees linear convergence for a sufficiently small stepsize, and empirically verified its superior performance, an explicit expression for the actual convergence rate or an explicit bound on the stepsize was not established.  In addition, while a superior linear convergence was empirically demonstrated compared to its first order optimization counterparts, there was no analytical justification for this observation. In this paper, we set out to provide a deeper analysis of the convergence of Network-GIANT and provide partial answers to some of the above unanswered questions. 

\subsection*{Contributions} 
Our main contributions are highlighted below.
\begin{itemize}[leftmargin=*, label = \(\circ\)]
\item We first provide a global linear convergence result with an explicitly computable rate for \(\ngiant\) for a sufficiently small step size $\eta$. 
This is achieved in a similar vein to \cite{li_na_gradient_tracking},  by establishing a linear inequality for a three-dimensional system consisting of the norms of consensus error, gradient tracking error, and the optimality gap error. We show that the linear convergence rate is explicitly given by the spectral radius of an associated $3 \times 3$ matrix dependent on the step size $\eta$, Lipschitz continuity ($L$) and strong convexity parameter ($\mu$) of the local cost functions, and the spectral norm ($\sigma$) of the doubly-stochastic consensus matrix of the underlying graph connecting the nodes. 
\item We also provide an explicit (albeit conservative) upper bound on the step size below which \(\ngiant\) is guaranteed to converge linearly (i.e., the aforementioned spectral radius is less than $1$).
\item Finally, we derive a mixed linear-quadratic inequality-based upper bound for the optimality gap norm, under the condition that the error ($\gamma$) due to the Hessian approximation is sufficiently small. Exploring this result further, we prove a local asymptotic convergence rate of $\left(1 - \eta(1-\frac{\gamma}{\mu})\right)$ as the algorithm approaches the optimum, and an associated less conservative upper bound on the stepsize,  which {\em analytically} explains the superior convergence rate of \(\ngiant\) compared to its first-order counterparts.
\end{itemize}
We illustrate the above findings through a distributed logistic-regression example with a reduced CovType dataset over regular expander-type and Erd\H{o}s–R\'enyi random graphs with varying degrees of connectivity. Furthermore, we refer the readers to \cite{alessio_net_giant} for a thorough comparison of \(\ngiant\) with other state-of-the-art fully distributed second-order algorithms.

\section{Problem Formulation}
We consider a fully distributed unconstrained optimization problem over a communication network, modeled by a graph, and is denoted by $\mathcal{G} = (\mathcal{N},\mathcal{E})$.  Here $\mathcal{N} = \{1, 2, ..., N\}$ and $\mathcal{E} \subseteq \mathcal{N} \times \mathcal{N}$ are the set of nodes and the set of bidirectional edges connecting the nodes, respectively.
    \begin{assumption}
        \label{assum:graph}
        We assume the following blanket properties about the network:
        \begin{itemize}
            \item The communication network is assumed to be an undirected, connected, and time-invariant graph.
            
            \item Each node/agent \(i \in \mathcal{N}\), can communicate only with its 1-hop neighbors.

            \item The network can be equivalently described by a weight matrix $W \in \mathbb{R}^{N \times N}$, where the element $w_{ij}$ is positive if there is an edge $(i, j) \in \mathcal{E}$ and zero otherwise.

            \item The consensus matrix $W$ is chosen to be symmetric and doubly stochastic, which implies that $W \mathbf{1} =\mathbf{1}$ and 
            $\mathbf{1}^T W = \mathbf{1}^T$ (here $\mathbf{1} \in \mathbb{R}^{N \times 1}$ is a column vector of all 1's), and the eigenvalues of $W$ lie in $(-1, 1]$.
        \end{itemize}
    \end{assumption}
     It is possible to construct this matrix in a distributed way, for example, using the Metropolis algorithm \cite{xiao2007distributed}. We will use the following conventions for vector and matrix norms:
    the vector norm $\norm{.}$  represents the $2$-norm, whereas for a matrix, $\norm{.}$ denotes the Frobenius norm. The matrix $2$-norm (or spectral norm), if used, will be denoted by $\norm{.}_2$. It is a well-known fact 
    that $\sigma = \norm{W - \frac{1}{N} \mathbf{1}\mathbf{1}^T}_2$ satisfies 
    $0 < \sigma < 1$.

 The unconstrained optimization problem is a fully distributed version of the Federated optimization problem considered in \cite{wang2018giant}:
The unconstrained global optimization problem is defined as 
\begin{equation}
    f(x^\star) = \min_{x \in \mathbb{R}^n} \left\{ f(x) = \frac{1}{N} \sum_{i=1}^N f_i(x) \right\},
    \label{eq:problem_formulation/optimization}
\end{equation}
where $f_(x)$ denotes the local loss function for the $i$-th node. 
\begin{remark} 
In typical distributed machine learning applications, the local loss function is often chosen as 
\begin{equation}
    f_i(x) = \frac{1}{m} \sum_{j=1}^{m} l_{ij}(x^T a_{ij}) + \frac{\tilde \lambda}{2} \norm{x}^2.
    \label{eq:problem_formulation/f_i}
\end{equation}
where 
each node $i \in \mathcal{N}$ owns $m$ data samples $\{ a_{ij} \}$ $j \in \{1, ..., m\}$, each associated with a convex, twice differentiable, and smooth loss function $l_{ij}(\cdot)$. The overall cost function at node $i$ is given by the sum of the local empirical error and a regularization term. Typical examples of such loss functions are 
linear regression or logistic regression problems. 
\end{remark}
We make the following standard assumption regarding the local loss functions $f_i(\cdot)$:
\begin{assumption}
\label{assumption_1}
 The local objective $f_i(\cdot)$ is $\mu$-strongly convex and $L$-smooth, i.e. there exists a constant $L$ such that $\forall x, x' \in \mathbb{R}^n$, and for all $i \in \{1, 2, \ldots, N \}$,  $\norm{\nabla f_i(x) - \nabla f_i(x') }_2 \leq L \norm{x - x'}$, and $\mu \mathbb{I} \preceq \nabla^2 f_i(x) \preceq L \mathbb{I} $, where $\mathbb{I}$ denotes an identity matrix of size 
$N\times N$. Note that this assumption implies that the global objective function $f(x)$ is also $\mu$-strongly convex and $L$-smooth. 

\end{assumption}
 Assumption \ref{assumption_1} guarantees the uniqueness of the minimizer of \eqref{eq:problem_formulation/optimization} and is standard in second-order optimization, and can often be guaranteed by the addition of a regularizer to the cost function if it is not strictly convex.

Let us define the matrices $\mathbf{x}_k, \mathbf{s}_k,  \nabla_k 
\in \mathbb{R}^{N \times n}$ by

\begin{align}
& \mathbf{x}_k  =  \begin{bmatrix} x_1^k  \\ \vdots \\ x_N^k \end{bmatrix}, \, \mathbf{s}_k =   \begin{bmatrix} s_1^k \\  \vdots \\ s_N^k \end{bmatrix}, \, \nabla_k =  \begin{bmatrix} \nabla f_1(x_1^k)  \\ \vdots \\ \nabla f_N(x_N^k) \end{bmatrix}. \label{globalvar} \end{align}
In order to avoid Kronecker-product-based cumbersome notations, we define the relevant vectors in the row-vector format throughout the paper. We assume that each node $i$ keeps a local copy $x_i^k \in \mathbb{R}^{1 \times n}$ of the global state vector 
$\mathbf{x}_k$ at the $k$-th iteration, and a vector $s_i^k \in \mathbb{R}^{1 \times n}$ that tracks the global average gradient 
$\nabla f(\mathbf{x}_k) = \frac{1}{N} \sum_{i=1}^N \nabla f_i(x_i^k)$, where $\nabla f_i(x_i^k) \in \mathbb{R}^{1 \times n}$, denotes the local gradient at the node $i$.  We also define the global Hessian as 
$\nabla^2 f(\mathbf{x_k}) = \frac{1}{N} \nabla^2f_i(x_i^k)$, and with a slight abuse of notation 
$\nabla^2 f(\tilde x) = \nabla^2 f(\mathbf{x_k}) \mid_{x_i^k = \tilde x^T}$ for further analysis.

	\section{Algorithm description}
	
	Define the local Hessian by \(\tilde{H_i}^k = \nabla^2 f_i(x_i^k)\) at the $i$-th node at the $k$-th iteration, and set the initial condition \(s_i^0 = \nabla f_i(x_i^0)\). The Network-GIANT algorithm is given by the two following consensus-based updates at each node $i$ at iteration $k$, one for the parameter vector $x_i^k$, and the other one being the familiar update for the gradient-tracking vector $s_i^k$ \cite{li_na_gradient_tracking}, as given below for all \(k \in \Nz\):\footnote{The set \(\Nz\) denotes the set of all natural numbers.}
    \begin{align}\label{netgalg}
        \begin{cases}
        x_i^{k+1} = \sum_{j=1}^n w_{ij} x_i^k  - \eta s_i^k \big(\tilde{H_i}^k\big)^{-1}\\
        s_i^{k+1} = \sum_{j=1}^n w_{ij} s_i^k + \nabla f_i(x_i^{k+1}) - \nabla f_i(x_i^k).
        \end{cases}
    \end{align}
    Here $\tilde{H_i}^k = \nabla^2 f_i(x_i^k)$ is invertible due to the strong convexity assumption on $f_i(x)$. The stepsize is denoted by $\eta \in (0,1)$, which needs to be chosen carefully.
    \begin{remark}\label{rem:further approx}
        Note that although the above algorithm requires an inverse local Hessian computation at each node, there are computationally efficient iterative approximate Hessian inverse calculation methods, e.g., a variation of the standard conjugate-gradient algorithm called LiSSA \cite{JMLR_hazan}, 
        or sketching-type approximation methods \cite{ref:SW-AG-MWM-18}. There are also similar efficient methods for computing the product of the inverse Hessian and gradients. These techniques can be used to alleviate the computational burden for Hessian inversion at the nodes in case of high dimensional problems.
    \end{remark}

    Using the matrix notation from \eqref{globalvar}, one can write the above algorithm compactly as follows:
    \begin{align}
   & \mathbf{x}_{k+1} = W \mathbf{x}_k - \eta \mathbf{y}_k \nonumber \\
   & \mathbf{s}_{k+1} = W \mathbf{s}_k  + \nabla_{k+1} - \nabla_k, \; \mathbf{s}_0 = \nabla_0.
        \label{netgalg_global}
    \end{align}
    Here $\mathbf{y}_k = [y_1^k \, y_2^k \, \ldots, y_N^k]^T$, where $y_i^k = s_i^k {\nabla^2 f_i(x_i^k)}^{-1} = s_i^k \tilde{H_i}^k $. For the subsequent convergence analysis of the algorithm given in \eqref{netgalg}, we also define the following average quantities  $\bar x_k = 
    \frac{\mathbf{1}^T \mathbf{x}_k}{N} = 
    \frac{1}{N} \sum_{i=1}^N x_i^k, \bar s_k = \frac{\mathbf{1}^T \mathbf{s}_k}{N}= \frac{1}{N} \sum_{i=1}^N s_i^k$ and $ g_k = \nabla f(\mathbf{x}_k)  = \frac{\mathbf{1}^T \nabla_k}{N}$. Using a similar inductive proof as in \cite{li_na_gradient_tracking}, it follows that $ \bar s^k = g_k$, that is, the average gradient tracking vector tracks the average global gradient at all steps, when initialized as $\mathbf{s}_0 = \nabla_0 $.

	\section{Convergence Analysis}
	In this section, we provide two main results regarding the convergence of the Network-GIANT algorithm given by 
    \eqref{netgalg_global}. First, under Assumptions 1(a) and 1(b), we establish that for a sufficiently small step size $\eta$, the 
	Network-GIANT algorithm enjoys a linear convergence rate, implying that the optimality gap $||\bar x^k - x^{\star}||$ decays exponentially with the iteration number $k$. We can provide an explicit expression for an upper bound $\bar \eta$ such that for $\eta < \bar \eta$, this result holds. In the second result, with additional assumptions on the Lipschitz continuity of the Hessians of the local loss functions and a small enough approximation error bound for the error between the true Hessian of $f(x)$, and the one given by the harmonic mean of the local Hessians, we prove a mixed linear-quadratic inequality based upper bound  for the optimality gap in the
    Network-GIANT algorithm.

\subsection{Linear Convergence Analysis of Network-GIANT}
    Before we proceed further, we quote the following results that have been established in \cite{li_na_gradient_tracking} under Assumption \ref{assumption_1}.
    We also define the three error terms: $\norm {\mathbf{x}_{k} - \mathbf{1} \bar x_k}$, $\norm{\mathbf{s}_{k} - \mathbf{1} g_k}$, and 
    $\norm{\bar {x}_{k} - x^{\star}}$, which correspond to the norms of the {\em consensus error}, {\em gradient tracking error}, and the 
    {\em optimality gap}, respectively.
    \begin{lemma} \label{lemmaslina}
    Suppose that Assumption 1 holds. 
    \begin{enumerate}[leftmargin=*, label = (\ref{lemmaslina}-\alph*)]
        \item For \(t \in \mathbb{N}\),  we have
        \begin{align}
       & \norm{ \nabla_k - \nabla_{k-1}} \leq L \norm{\mathbf{x}_{k} - \mathbf{x}_{k-1}}, \\
       & \norm{ g_k - g_{k-1}} \leq \frac{L}{\sqrt{N}} \norm{\mathbf{x}_{k} - \mathbf{x}_{k-1}}.
        \end{align}

        \item Recall that \(\mathbf{(s_k)_{k \in \mathbb{N}}}\) denotes the gradient-tracker.  Then,
        \begin{align}
          & \norm{\mathbf{s}_{k}} \leq  \norm{\mathbf{s}_{k} - \mathbf{1} g_k}  + L  \norm{ \mathbf{x}_{k} - \mathbf{1}  x^{\star}}, \\ \nonumber
          &  \leq \norm{\mathbf{s}_{k} - \mathbf{1} g_k} + L \norm {\mathbf{x}_{k} - \mathbf{1} \bar x_k} + L \sqrt{N} \norm{\bar x_k - x^{\star}}.
          \end{align}
    \end{enumerate}
     
    \end{lemma}
     We will also need another result concerning the global linear convergence of a damped Newton method, stated as follows.
     \begin{lemma}
    Consider a column vector $z \in \mathbb{R}^n$, and a function $u(\cdot): \mathbb{R}^n \longrightarrow \mathbb{R}$, satisfying the smoothness and convexity assumptions in Assumption \ref{assumption_1}. Let $z^{\star}$ be the unique minimum of $u(z)$. Consider the following damped Newton iterative update: for \(t \in \mathbb{N}\), we have
     \begin{align} \label{lem:update}
         z^{+} =  z - \eta [\nabla^2 u(z)]^{-1} \nabla u(z), \; 0 < \eta < 1.
     \end{align}
     If $\eta \leq \frac{\mu}{L}$, then 
     \begin{align}
         \norm{z^{+} - z^{\star}} \leq (1-\eta \frac{\mu}{L}) \norm{z- z^{\star}}.
     \end{align}
         \end{lemma}
         \begin{proof}
             Define $H_z \coloneqq [\nabla^2 u(z)]^{-1}  \int_0^1 \nabla^2 u (z^{\star} + \lambda (z-z^{\star}) d \lambda$. Since \(z^{*}\) is the minimizer of \(u\), \(\nabla u(z^{\star}) = 0\).
             Then one can show that 
             \begin{align}
                 & [\nabla^2 u(z)]^{-1} \nabla u(z)  = [\nabla^2 u(z)]^{-1} (\nabla u(z) - \nabla u(z^{\star})) \nonumber \\
                 & = \bigg( [\nabla^2 u(z)]^{-1} \int_0^1 \nabla^2 u (z^{\star} + \lambda (z-z^{\star}) d \lambda \bigg)
                 (z - z^{\star}) \nonumber \\
                 & = H_z (z - z^{\star}), 
             \end{align}
             which implies
             $z^{+} - z^{\star} = ( \mathbb{I} - \eta H_z) (z - z^\star)$. Appealing to Assumption \ref{assumption_1}, and using $\frac{\mu}{L} \mathbb{I} \leq H_z \leq \frac{L}{\mu} \mathbb{I}$, it follows that 
             $$(1 - \eta \frac{L}{\mu}) \mathbb{I} \leq (\mathbb{I} - \eta H_z) \leq (1-\eta \frac{\mu}{L}) \mathbb{I}.$$
             Therefore, if $\eta \leq \frac{\mu}{L}$, we obtain 
             $0 \leq \mathbb{I} - \eta H_z \leq (1-\eta \frac{\mu}{L})\mathbb{I}$, implying 
             $\norm{\mathbb{I} - \eta H_z}_2 \leq (1 - \eta \frac{\mu}{L})$, which completes the proof.
         \end{proof}
         \begin{remark}
             It is well known that a standard Newton's method, initialized with a damped stage where the step size $\eta$ is chosen by a backtracking line search, will eventually allow $\eta \rightarrow 1$, as the gradient norm falls below a certain threshold, resulting in a local quadratic convergence of the Newton's method with $\eta=1$ \cite{boydbook}. However, the damped Newton's method with a fixed step size $\eta$ also enjoys global linear convergence for a sufficiently small $\eta$ for strongly convex and smooth functions, as shown in the above proof. Interestingly, the global linear convergence of a slightly more general damped Newton's method is left as an exercise 
             in \cite{polyak1987introduction}, but has been established as a corollary with a more general Hessian stability property in 
             \cite{karimireddy2018global}.
         \end{remark}
    Using the above two lemmas, one can prove the following:
    \begin{theorem}
        Consider the NETWORK-GIANT algorithm given by \eqref{netgalg_global}, implemented with a doubly stochastic consensus matrix $W$ with a spectral norm $\sigma = \norm{W-  \frac{1}{N} \mathbf{1}\mathbf{1}^T}_2 < 1$. Define the consensus error(CE), the gradient tracking error (Grad-T), and the optimality gap error (OE)  (at the $k$-th iteration)  as $\mathbf{x}_{k} - \mathbf{1} \bar x_{k}$,
        $\mathbf{s}_{k} - \mathbf{1} g_{k}$, and $\bar {x}_{k} - x^{\star}$ respectively. 
        Assuming $\eta \leq \frac{\mu}{L}$, and that Assumption 2 holds, we can show the following inequality:
        \begin{align}
        & \begin{bmatrix} \norm{\mathbf{x}_{k+1} - \mathbf{1} \bar x_{k+1}} \\ \norm{\mathbf{s}_{k+1} - \mathbf{1} g_{k+1}} \\
               \sqrt{N} \norm{\bar {x}_{k+1} - x^{\star}}   \end{bmatrix} \leq   G(\eta)  \begin{bmatrix} \norm{\mathbf{x}_{k} - \mathbf{1} \bar x_{k}} \\ 
                \norm{\mathbf{s}_{k} - \mathbf{1} g_{k}} \\
                \sqrt{N} \norm{\bar {x}_{k} - x^{\star}}   \end{bmatrix},           
        \label{norminequality}
        \end{align} 
    where the matrix \(G(\eta)\) is given by
    \begin{align*}
                & G(\eta) = \begin{bmatrix}
                \sigma+ \eta \frac{L}{\mu} & \frac{\eta}{\mu} & \eta \frac{L}{\mu} \\
                2L + \eta \frac{L^2}{\mu} & \sigma+ \eta \frac{L}{\mu} & \eta\frac{L^2}{\mu} \\
                \eta \frac{L}{\mu} & \frac{\eta}{\mu} & 1- \eta \frac{\mu}{L} 
        \end{bmatrix}.
    \end{align*}
    
    \label{lin_convergence}    
    \end{theorem}
    \begin{proof}
        See Appendix \ref{appen:th_1} for a proof.
    \end{proof}
    It is clear from \eqref{norminequality} that the three error norms will go to zero exponentially fast as long as the spectral radius 
    $\rho(G(\eta)) < 1$. Since $G(\eta)$ is a matrix with all positive elements, from the Perron-Frobenius theorem, the largest eigenvalue is real and positive. In the next theorem, we provide an upper bound on the step size $\eta$, ensuring that the spectral radius $\rho(G(\eta)) < 1$. 

\begin{theorem}\label{th:linear rate}
   Suppose the stepsize  $\eta < \bar \eta = \frac{(1-\sigma)^2} {2(2-\sigma)\left(\kappa + \kappa^3 \right)}$, where 
   $\kappa= \frac{L}{\mu}$ is the condition number of the global Hessian. Then it follows that $\rho(G(\eta)) < 1$. 
    \end{theorem}

    \begin{proof}
        A proof is provided in Appendix \ref{appen:th_2}.
    \end{proof}
    \begin{remark}
      In \cite{li_na_gradient_tracking}, which established a linear convergence result for the gradient tracking based first-order distributed optimization algorithm, the authors showed that (see  Lemma 2 in \cite{li_na_gradient_tracking}) when $\eta < \frac{1}{L}$ (using the notation of our paper), the spectral radius of the corresponding matrix 
      $$\bar G(\eta) =\begin{bmatrix}
        \sigma  & \eta  & 0 \\
        2L + \eta L^2 & \sigma+ \eta L & \eta L^2 \\
        \eta L & 0 & 1- \eta \mu  \end{bmatrix}    $$
        is upper bounded, and for a specific value of $\eta = \frac{\mu}{L^2} (\frac{1-\sigma}{6})^2$, the spectral radius 
    $\rho(\bar G(\eta))  < 1$. Using a similar analysis as in the proof of Theorem 2 in our paper, one can obtain a tighter bound on $\eta $ compared to $\frac{1}{L}$, which is given by 
    $\eta < \tilde \eta = \frac{-(3-\sigma)+\sqrt{(3-\sigma)^2 + 4 (1-\sigma)^2 (1+\kappa)}}{2 L (1 + \kappa)}$, for which, the spectral radius 
    $\rho(\bar G(\eta))  < 1$. 
    \end{remark}
\begin{remark}
Note that the upper bounds on the step sizes for both the Network-GIANT algorithm and gradient tracking based algorithm in \cite{li_na_gradient_tracking}, given by $\bar \eta$, and $\tilde \eta$ are simply sufficient conditions, as the evolution of the three error norms in \eqref{norminequality} is given by inequalities rather than equalities. Thus, while choosing an $\eta$ below the corresponding upper bounds will definitely guarantee convergence, it may be possible to obtain faster convergence with a larger step size. For example, 
for a graph with $\sigma=0.7$, and $L=5, \mu=1$ (i.e., a mild condition number of $5$), it can be shown through numerical computation for $0 < \eta < \tilde \eta$, that the smallest value of 
$\rho(\bar G(\eta))$ is $0.9954$ for the distributed gradient tracking algorithm in 
\cite{li_na_gradient_tracking}.
Similar conservative convergence rates, based on Theorems 1 and 2, hold for the Network-GIANT algorithm.
\end{remark}
\begin{remark}
It can be easily verified that there is no universal ordering between $\bar \eta$ and $\tilde \eta$, as depending on the specific parametric values of $L, \mu$, and $\sigma$, either $\bar \eta <\tilde \eta $ or $\tilde \eta <\bar \eta $ may be possible. 
\end{remark}

\begin{remark}
    \label{rem:global sigma}
    A global parameter such as the spectral radius \(\sigma\) can be estimated online in a fully distributed manner; see, for example, \cite{ref:RA-GS-DVD-CS-KHJ-etal-14,ref:AY-15,ref:YZ-SL-JW-20}, which rely on local data and consensus-driven protocols.
\end{remark}
    \subsection{A Mixed Linear-Quadratic Inequality for the Optimality Gap}
    Motivated by the conservativeness of the global linear convergence results in the previous subsection, we investigate why we obtain substantially faster convergence rates for Network-GIANT in numerical results than those predicted by Theorems 1 and 2.
    To this end, we show that under some further assumptions,  one can obtain a mixed linear-quadratic convergence rate of the optimality gap norm $\norm{\bar {x}_{k} - x^{\star}}$. Before we proceed, we define the `true' global Hessian (arithmetic mean of the local Hessians) and the `approximate' Hessian (given by the harmonic mean of the local Hessians), respectively, by 
    \begin{align*}
        \begin{cases}
           H_{tr}(\mathbf{x_k}) = \frac{1}{N}\sum_{i=1}^N \tilde H_i^k,\\
           H_{app}(\mathbf{x_k}) = \left( \frac{1}{N}\sum_{i=1}^N (\tilde H_i^k)^{-1} \right)^{-1},
        \end{cases}
    \end{align*}
    where \(\tilde{H_i}^k = \nabla^2 f_i(x_i^k)\) is the local Hessian of the \(i\)-th agent at \(k\)-th iteration.\footnote{Here the subscripts `tr' and `app' are shorthand notations for `true' and `approximate'.}

    We make the following two assumptions. The first assumption is standard in the analysis of Newton-type optimization algorithms, and the second assumption is on the tightness of the approximation of the global Hessian by the harmonic mean of the local Hessians.
    \begin{assumption}
        We assume that the Hessians of all the local cost functions are Lipschitz continuous, i.e.
        \begin{align}
            \norm {\nabla^2 f_i(x) - \nabla^2 f_i(y)} \leq \bar L \norm{x-y}, \forall i.
        \end{align}
        \label{hessianlip}
    \end{assumption}
    \vspace*{-0.2cm}
    \begin{assumption}\label{hessapprox}
        There exists a sufficiently small $\gamma$ such that 
        \begin{align}
        \norm{H_{tr}(\mathbf{x}) - H_{app}(\mathbf{x})} \leq \gamma, \; \forall \mathbf{x} \in \Rbb^{N \times n}. 
        \end{align}
    \end{assumption}
    \begin{remark}
        \label{rem:hessian approx}
        Note that Assumption \ref{hessapprox}  is quite standard in the literature on approximate Newton-based algorithms, including, for example, \cite{ye2021approximate} (centralized ML), and \cite{li2020communication} (decentralized ML with Network-DANE). It is generally difficult to derive useful explicit analytical bounds on $\norm{H_{tr}(\mathbf{x}) - H_{app}(\mathbf{x})}$. Nevertheless, in \cite{wang2018giant}, the authors empirically demonstrated that, for regression-based local loss functions of the form $f_i(x) = \frac{1}{m} \sum_{j=1}^{m} l_{ij}(x^T a_{ij}) + \frac{\lambda}{2} \norm{x}^2$,  when the data samples are distributed homogeneously across agents and the local Hessian approximations constructed from the local sample sets are sufficiently accurate, the arithmetic mean $H_{tr}(\mathbf{x})$ can indeed be approximated by the harmonic mean $H_{app}(\mathbf{x})$ with reasonably good accuracy. Indeed, in Section V (Numerical Analysis), we illustrate with numerical results that the global Hessian approximation error achieved via Network-GIANT can be made much smaller than the least eigenvalue of the true global Hessian when datasets across the various agents are statistically homogeneous (see Figure \ref{fig:hetero_hess_error}).     
    \end{remark}
	\begin{remark}
        Note also that the following linear-quadratic inequality result is not limited to Network-GIANT only, as any distributed approximate-Newton algorithm that can achieve a global Hessian approximation error less than $\mu$ will satisfy the mixed linear-quadratic inequality result stated below.
	\end{remark}
    \begin{theorem}\label{th:mixed LQ rate}
     Under Assumptions \ref{assumption_1}, \ref{hessianlip}, and \ref{hessapprox},  given $\gamma < \mu$, and 
     $\eta < 1$,  the optimality gap error norm $\norm{\bar {x}_{k+1} - x^{\star}}$ at iteration $k+1$   satisfies the following mixed linear-quadratic inequality:
     \begin{align}
        & \norm{\bar {x}_{k+1} - x^{\star}} \leq \left(1- \eta(1 -\frac{\gamma}{\mu})\right) \norm{\bar {x}_{k} - x^{\star}} \nonumber \\
        & \;\;\;\;\;\;\; + \frac{\eta \bar L}{\mu \sqrt{N}} \norm{\mathbf{x}_{k} - \mathbf{1} \bar x_{k}} \norm{\bar {x}_{k} - x^{\star}} 
         + \frac{\eta \bar L}{2 \mu} \norm{\bar {x}_{k} - x^{\star}}^2  \nonumber \\
       & \;\;\;\;\;\;\;  + \frac{\eta}{\mu \sqrt{N}} \norm{\mathbf{s}_{k} - \mathbf{1} g_{k}} 
        + \frac{\eta L}{\mu \sqrt{N}} \norm{\mathbf{x}_{k} - \mathbf{1} \bar x_{k}}.
        \label{lin_quad_rate}
     \end{align}
     \label{lin_quad_conv}
    \end{theorem}
 \begin{proof}
        Please refer to Appendix \ref{appen:th_3} for a proof.
    \end{proof}

    It was illustrated with a comprehensive set of numerical results in 
    \cite{alessio_net_giant} that Network-GIANT outperforms first-order optimization algorithms based on gradient tracking and several other distributed optimization methods using different types of Hessian approximation techniques, such as Network-DANE \cite{li2020communication}, and various other Network-Newton methods, including the one presented in \cite{mokhtari2016network}. While it is generally difficult to compare exact convergence rates of various distributed Newton type methods, Theorem 
    \ref{lin_quad_conv} explains why Network-GIANT provides a faster linear convergence rate than the gradient tracking based first order algorithm, e.g., in \cite{li_na_gradient_tracking}. The next result establishes this fact formally:
    \begin{theorem}
        \label{th:linear-quadratic rate}
        Consider the consensus update iteration in Theorem \ref{th:linear rate}, along with the iterations of the optimality gap error recursion from \eqref{lin_quad_rate}, denoted by \(\optgaperr_k = \norm{\bar {x}_{k} - x^{\star}}\) for each \(k\). Let Assumptions \ref{assumption_1}, \ref{hessianlip}, and \ref{hessapprox} hold. Recall that \(\kappa= \frac{L}{\mu}\). 
        If \(\gamma < \mu\), then we have
        \begin{align}\label{eq:geo_rate}
        \limsup_{k \ra + \infty} \bigg(\frac{\log \optgaperr_k}{k+1} \bigg) \le \log \bigg(1 - \eta \Big(1 - \frac{\gamma}{\mu} \Big) \bigg). 
        \end{align}
    \end{theorem}
    A proof of Theorem \ref{th:linear-quadratic rate} is provided in Appendix \ref{appen:proof_linear-quad-rate}.
    \begin{remark}
        Note that  \eqref{eq:geo_rate} implies that the sequence \((\optgaperr_k)_{k \in \Nz}\) asymptotically decays geometrically at the rate  $\Big(1 - \eta \big(1 - \frac{\gamma}{\mu} \big) \Big)$. In other words, Theorem \ref{th:linear-quadratic rate} asserts that the mixed linear-quadratic convergence rate of \eqref{lin_quad_rate} is asymptotically dominated by the linear term in \eqref{lin_quad_rate}.
    \end{remark}
    \begin{remark}
        \label{rem:misc}
         When the hessian approximation error $\gamma$ is sufficiently small compared to the strong convexity constant $\mu$, then this linear convergence rate is approximately $(1-\eta)$, which is much faster than the corresponding rate of the gradient tracking based first order method, given by approximately $(1-\eta \mu)$ in the third diagonal entry of $\bar G(\eta)$.
    \end{remark}
    A further local convergence analysis near the neighborhood of the origin $(0,0,0)$ for the three error norms given by 
    $\tilde e^c_k = \norm{\mathbf{x}_{k} - \mathbf{1} \bar x_{k}}$, $\tilde e^g_k = \norm{\mathbf{s}_{k} - \mathbf{1} g_{k}}$, and 
    $\optgaperr_k$ defined in Theorem \ref{th:linear-quadratic rate} can be done by combining the linear inequalities for 
    $\tilde e^c_k$ and $\tilde e^g_k$ from \eqref{norminequality} and the mixed linear-quadratic inequality \eqref{lin_quad_rate}, which yields the following  system of nonlinear recursive inequalities:
    \begin{align}
        & \begin{bmatrix} \tilde e^c_{k+1}\\ \tilde e^g_{k+1} \\ \sqrt{N}\optgaperr_{k+1} \end{bmatrix} \leq
        \hat G(\eta) \begin{bmatrix} \tilde e^c_{k}\\ \tilde e^g_{k} \\ \sqrt{N}\optgaperr_{k} \end{bmatrix} \nonumber \\
        &  + \begin{bmatrix} 0 \\ 0 \\ 1 \end{bmatrix} \left(\frac{\eta \bar L}{\mu} \tilde e^c_k \optgaperr_{k}
         + \frac{\eta \bar L}{2 \mu} \sqrt{N} \optgaperr^2_{k} \right), \label{sys_nonlinear} \end{align}
         where $$ \hat G(\eta)  = \begin{bmatrix}
                \sigma+ \eta \frac{L}{\mu} & \frac{\eta}{\mu} & \eta \frac{L}{\mu} \\
                2L + \eta \frac{L^2}{\mu} & \sigma+ \eta \frac{L}{\mu} & \eta\frac{L^2}{\mu} \\
                \eta \frac{L}{\mu} & \frac{\eta}{\mu} & 1- \eta (1 - \frac{\gamma}{\mu})
        \end{bmatrix}  $$           
  By taking partial derivatives of the right hand side of the inequality \eqref{sys_nonlinear} with respect to the three error norms, one can construct the Jacobian at the origin as exactly $\hat G(\eta)$. Therefore, the local convergence rate to the origin from within a small neighborhood of it is given by the spectral radius of $\hat G(\eta)$. By conducting a similar analysis to the proof of Theorem \ref{th:linear rate}, we can prove the following result on the local convergence rate to the origin of the three error norms (details are omitted due to space reasons):
    \begin{theorem}
        \label{th:combine analysis} The system of nonlinear inequalities \eqref{sys_nonlinear} is asymptotically locally convergent to the origin   if $\rho(\hat G(\eta)) < 1$, which can be guaranteed by choosing 
        \begin{align}
        \eta  <  \hat \eta = \frac{(1-\sigma)^2} {2(2-\sigma)} \left[ \frac{\omega}{\omega \kappa + \kappa^2} \right], \end{align} where 
        $\omega = (1-\frac{\gamma}{\mu})$. 
    \end{theorem}
    
    \begin{remark}
        \label{rem:combine}
        One can show that the threshold $\hat \eta$ is strictly greater than $\bar \eta$, (the upper bound of \(\eta\) obtained in Theorem 2), if $\omega > \frac{1}{\kappa}$, or 
        $\gamma < \mu (1 -\frac{\mu}{L})$. This implies that for local convergence to the origin for the system of three error norms, one can have a larger range of $\eta$, provided the Hessian approximation error $\gamma < \mu (1 -\frac{\mu}{L}) < \mu$. Hence, $\hat \eta$ is also an upper bound for the choice of $\eta$ for the asymptotic (local) linear convergence rate of 
        $\big(1-\eta(1 - \frac{\gamma}{\mu})\big)$ for the optimality gap to zero. 
    \end{remark}
    \begin{remark}[Time-varying or directed graphs]
        While the above theoretical results were presented for a time-invariant undirected graph, extensions to time-varying or directed graphs are certainly possible. For example, using the $B$-connected graph sequence assumption of time-varying double stochastic consensus matrices in \cite{nedic_tv_siam}, linear convergence rates can also be established for Network-GIANT. In the case of time-invariant directed graphs represented by row-stochastic matrices, we can leverage the techniques from \cite{ref:CX-VSX-18} to establish linear convergence for Network-GIANT. The extension to the more complex case of time-varying directed graphs with Network-GIANT will be investigated later.
    \end{remark}
    In the next section, we will illustrate the above results with numerical illustrations using a logistic regression-based binary classification task with the {\em CovType} data set. We will compare Network-GIANT and 
    the gradient tracking based first order method from \cite{li_na_gradient_tracking}, as well as a Nesterov-acceleration based gradient tracking algorithm ACC-NGD-SC \cite{ref:GQ-NL-19}.  To avoid repetition with \cite{alessio_net_giant}, we do not include comparative performance results for other second order approximate Newton algorithms. 
    \section{Numerical Analysis}
    In this section, we empirically investigate several important properties of Network-GIANT. We begin by verifying its asymptotic convergence behavior and demonstrate that, provided the Hessian approximation error is sufficiently small, the optimality-gap error converges asymptotically at a rate of \(1 - \stsz\). We then study the scalability of Network-GIANT with respect to the number of agents for Erd\H{o}s--R'enyi graphs. Next, we provide a detailed comparison with two first-order algorithms in terms of both convergence rate and wall-clock-time performance. Finally, we examine the impact of data heterogeneity on the convergence behavior of Network-GIANT and on the quality of its Hessian approximation. All experiments were conducted on a MacBook Pro M4 with a \(10\)-core processor, \(24\)GB of internal memory, and a clock speed of \(4.61\) GHz.

\subsection*{Experimental setup}For the experiments, we considered a binary logistic classification problem given by the distributed optimization problem
\begin{align}
    \label{eq:log_bin}
    \min_{\dvar \in \Rbb^{\dimdv}} \frac{1}{\agents}\sum_{i=1}^{\agents} \frac{1}{m}\sum_{j=1}^{m} \log\Big( 1 + \exp\big(-v^i_j(\dvar^{\top}u^i_j) \big)\Big) + \frac{\tilde \regu}{2} \norm{\dvar}^2,
\end{align}
where $\tilde \lambda$ is the regularizer chosen to be \(0.01\) across all experiments. We performed distributed classification on a reduced CovType dataset \cite{ref:DD-CG-19}, consisting of \(566602\) sample points and \(n = 10\) features. After random reshuffling to minimize bias, the samples $(u_j^i, v_j^i)$ were uniformly distributed across the $N$ nodes with ($m$ samples each). The variables \(\textbf{x}(0)\) were initialized uniformly randomly, and we imposed \(\grd f_i(x_i(0)) = s_i(0)\) for each \(i = \aset[]{1,2,\ldots, N}\).

The experiments were conducted on \(d\)-regular expander graphs and Erd\H{o}s-R\'enyi graphs \(\big(\text{denoted by} G(N,p)\big)\) \cite{erdos59a} with \(N=100\) nodes and different graph parameters. Here \(d\) and \(p\) denote the degree of the expander graph and the edge/connection probability associated with the Erd\H{o}s-R\'enyi graph, respectively, and \(\sigma\) is the spectral radius of these graphs. 
\begin{table}[htbp]
\centering
\begin{adjustbox}{max width=\linewidth}
\begin{tblr}{
  colspec = {l|cc|cc},
  row{1} = {font=\bfseries},
  hlines,
}
Graph Type 
  & \(d\) / \(p\) 
  & $\sigma$ 
  & \(d\) / \(p\)
  & $\sigma$ \\

Expander 
  & $d = 20$  & $0.9256$ 
  & $d = 72$ & $0.3178$ \\

Erd\H{o}s--R\'enyi 
  & $p = 0.07$  & $0.8103$ 
  & $p = 0.28$ & $0.4466$ \\
\end{tblr}
\end{adjustbox}
\caption{Spectral norm ($\sigma$) for different graph parameters.}
\label{tab:graph_type}
\end{table}

We assumed that each agent can exchange data with their \(1\)-hop neighbors. The consensus weight matrices of these graphs were constructed using \cite{ref:LX-SB-SJK-07}. 

\subsection*{Results and discussions}
In the following experiments, the true optimal solution \(\dvar^{\ast}\) and the corresponding optimal value \(\objfunc(\dvar^{\ast})\), used as baselines, are obtained from the centralized Newton-Raphson algorithm with backtracking line search.

Figure \ref{fig:n_giant_errors} displays the convergence of different errors, namely, the consensus error (CE), the gradient-tracking error (Grad-T), and the optimality gap error (OE), for two different graph configurations of \(d\)-regular expander graphs. To generate Figure \ref{fig:n_giant_errors},  we choose a well connected graph with \((d, \sigma) = (72, 0.3178)\) and a sparsely connected graph with \((d, \sigma)=(20, 0.9256)\), where \(\eta\) was chosen (optimized empirically) to be \(0.065\) and \(0.045\), respectively.
We observed that for small stepsize (\(\stsz\)) choices, the expander graph with \(d=72\) exhibits a higher degree of connectivity when compared to \(d = 20\), and therefore, a lower spectral norm \(\sigma\). Consequently, the consensus error dynamics and the gradient-tracking error attain small values earlier than the optimality gap error, which decays at a slower rate of \(1 - \stsz\frac{\mu}{L}\). For graphs with a lower degree of connectivity, faster convergence is not distinctly observed, as shown in Figure \ref{fig:n_giant_errors}.  
\begin{figure}[!htbp]
    \centering
    \includegraphics[scale = 0.15]{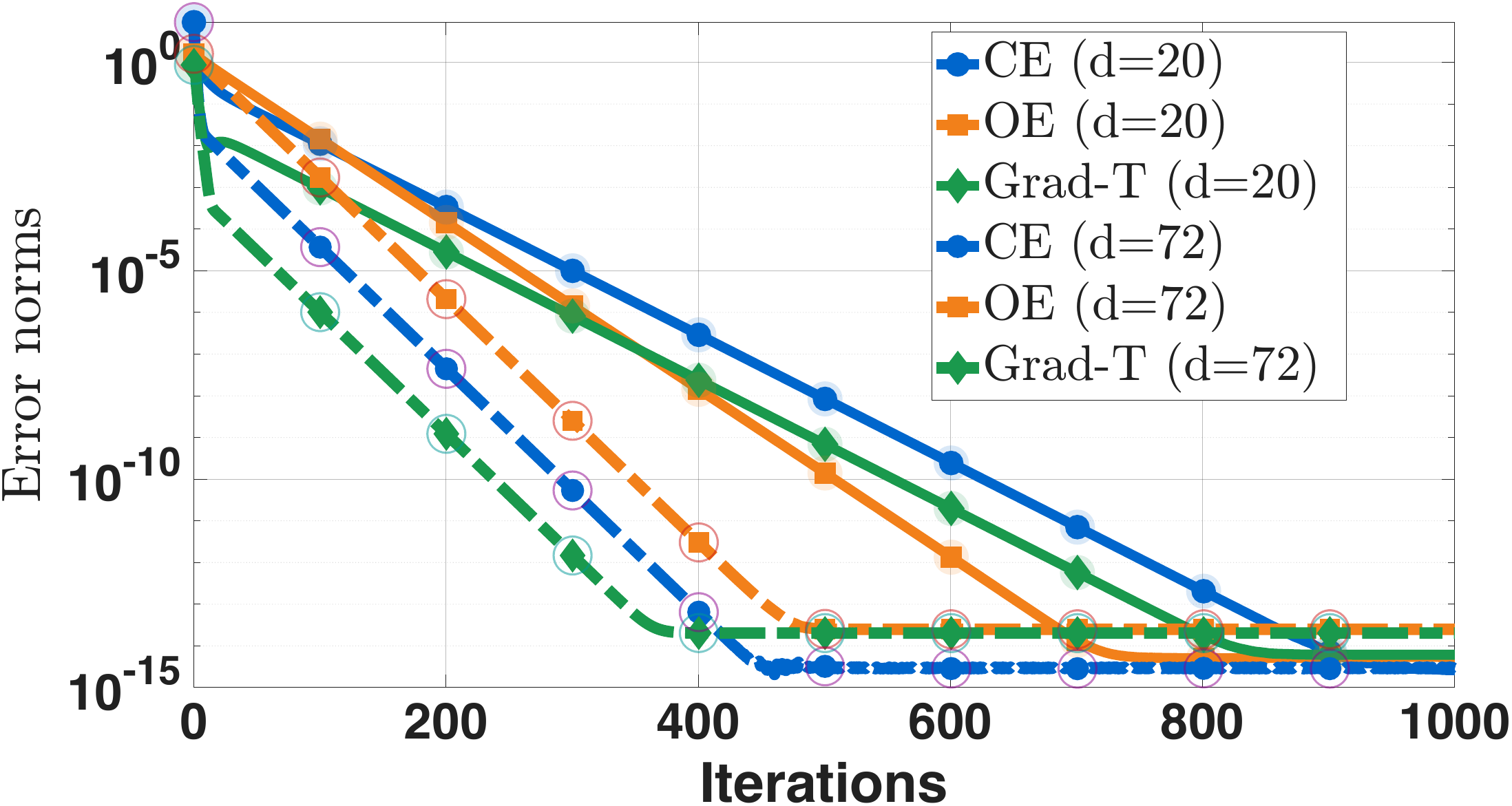}
    \caption{Errors associated with \(\ngiant\). Here `CE', `OE', and `Grad-T', are short forms for consensus error, optimality gap error, and gradient-tracking error, respectively.}
    \label{fig:n_giant_errors}
\end{figure}
We further observe that when the consensus and gradient tracking errors converge to sufficiently small values, the linear term in the optimality error gap dominates the quadratic term at a later stage of consensus when \(\bar{x}_k\) lies in a neighborhood of \(x^{\ast}\).
\subsubsection*{Empirical evidence for improved rate of convergence}
Figures \ref{fig:roc_ex} and \ref{fig:roc_er} illustrate that for small \(\stsz\), the ratio \(r_k\), defined by 
\begin{equation}\label{eq:roc}
    r_k \coloneqq \frac{\norm{\mathbf{x}_{k+1} - x^{\ast}}}{\norm{\mathbf{x}_{k} - x^{\ast}}} \le 1 - \stsz + \frac{\stsz \gamma}{\mu} \approx 1 - \stsz \text{ for each }k,
\end{equation}
provided that the Hessian approximation error \(\gamma\) is sufficiently small. Two observations are noteworthy here. First, we report that for both the graph models, namely \(d\)-regular expander graphs and Erd\H{o}s-R\'enyi graphs, \((r_k)_{k \in \Nz}\) converges asymptotically approximately to the limiting value of \(1 - \stsz\), provided \(\stsz\) and the Hessian approximation error are small. This is consistent with the assertion in Theorem \ref{th:linear-quadratic rate}. 
\begin{figure}[!htbp]
    \centering

    \begin{subfigure}{\textwidth}
        \centering

        \begin{subfigure}{\columnwidth}
            \includegraphics[scale = 0.18]{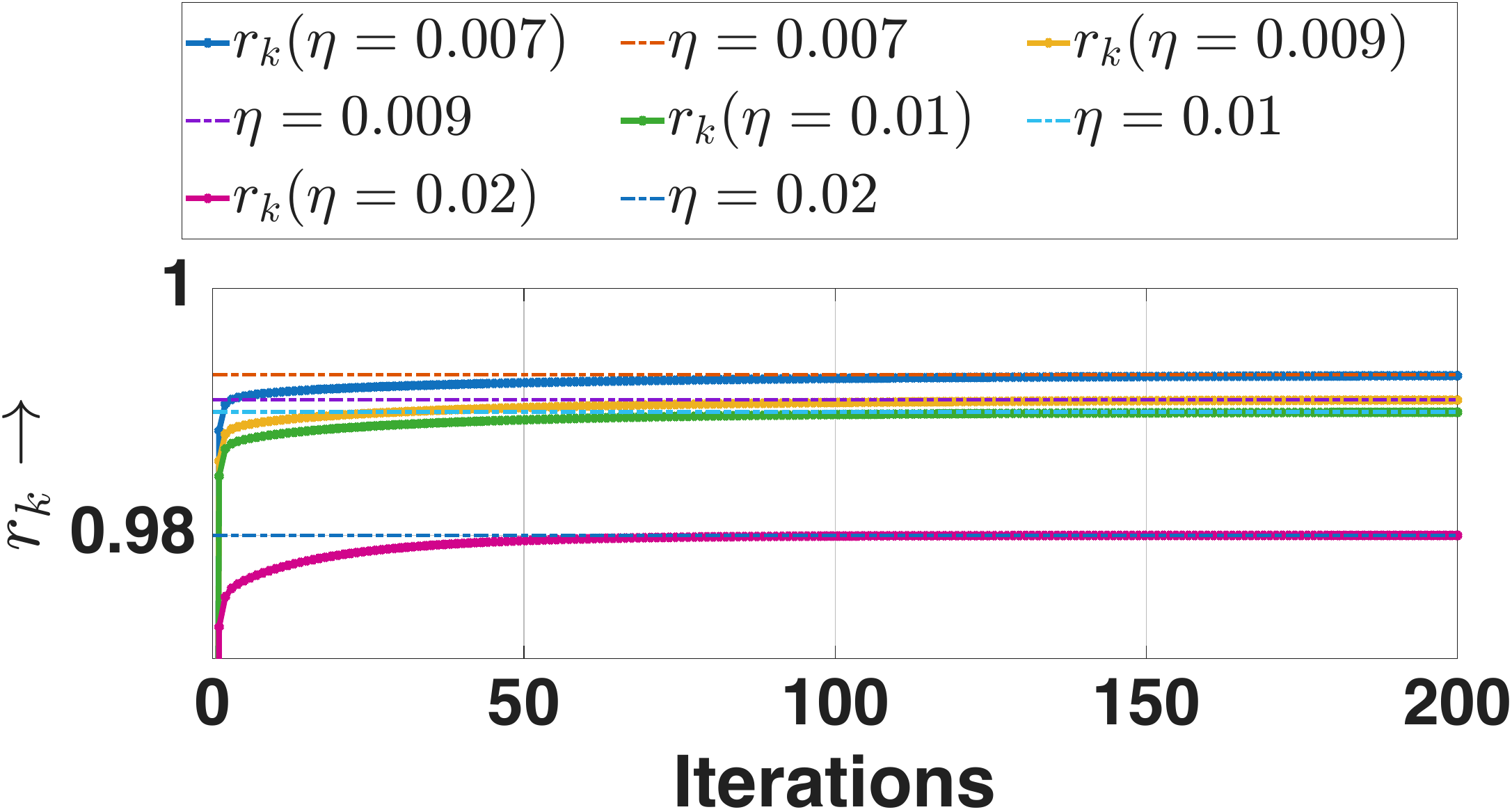}
            \label{subfig:d2_expander}
        \end{subfigure}\vspace{1.5mm}
        \begin{subfigure}{\columnwidth}
            \includegraphics[scale = 0.18]{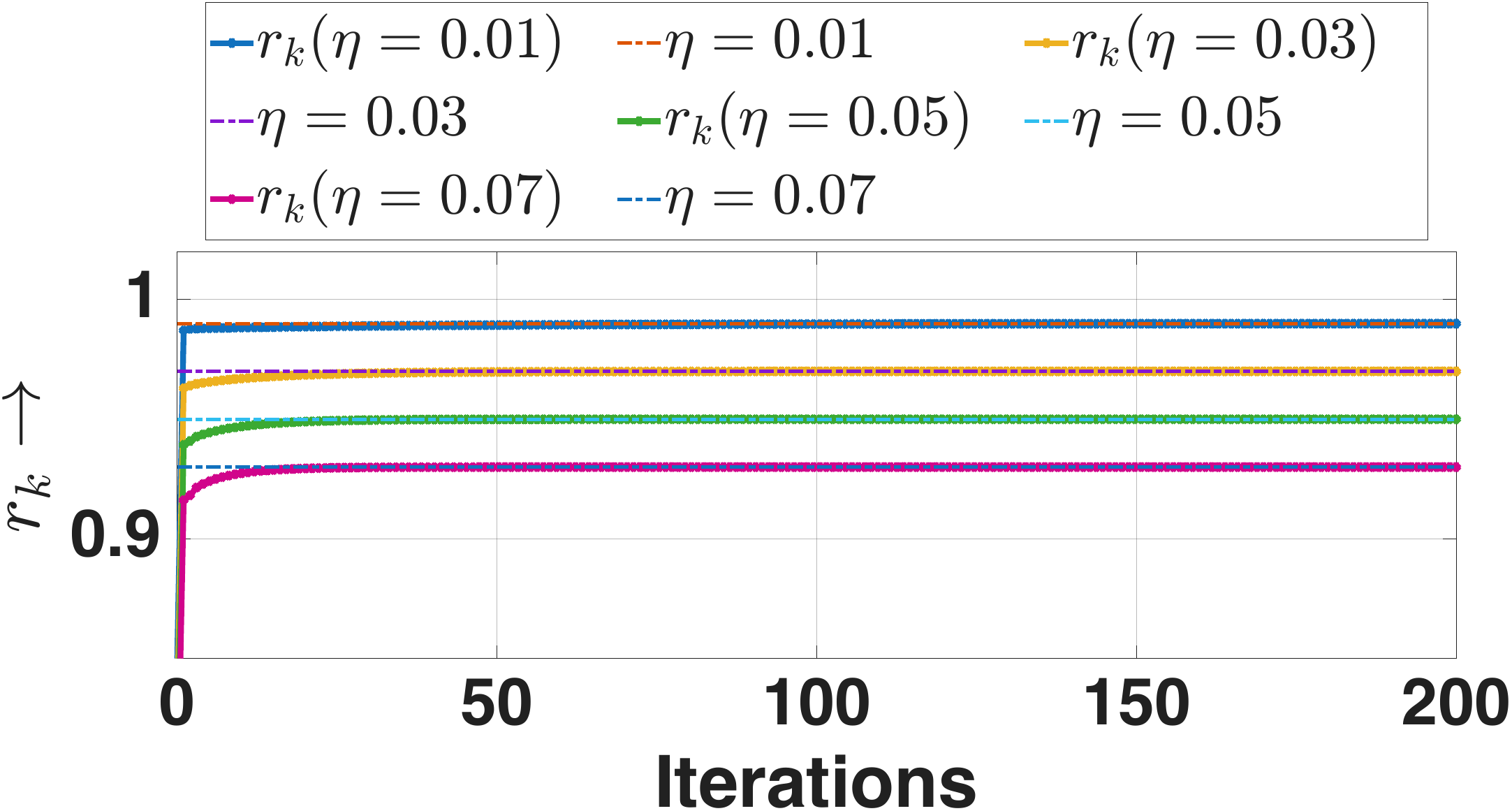}
            \label{subfig:d14_expander}
        \end{subfigure}
        \label{subfig:row1_expander}
    \end{subfigure}

    \caption{Rate of convergence \eqref{eq:roc} for \(d\)-regular  expander graph with \(d = 20\) (top) and \(d = 72\) (bottom).}
    \label{fig:roc_ex}
\end{figure}
For graphs with a higher degree of connectivity, a locally asymptotic approximate convergence rate of $1 -\stsz $ is seen in Figures \ref{fig:roc_ex} and \ref{fig:roc_er} (both bottom figures), even for \(\eta = 0.07\).
\begin{figure}[!htbp]
    \centering

    \begin{subfigure}{\textwidth}
        \centering
        \begin{subfigure}{\columnwidth}
            \includegraphics[scale = 0.18]{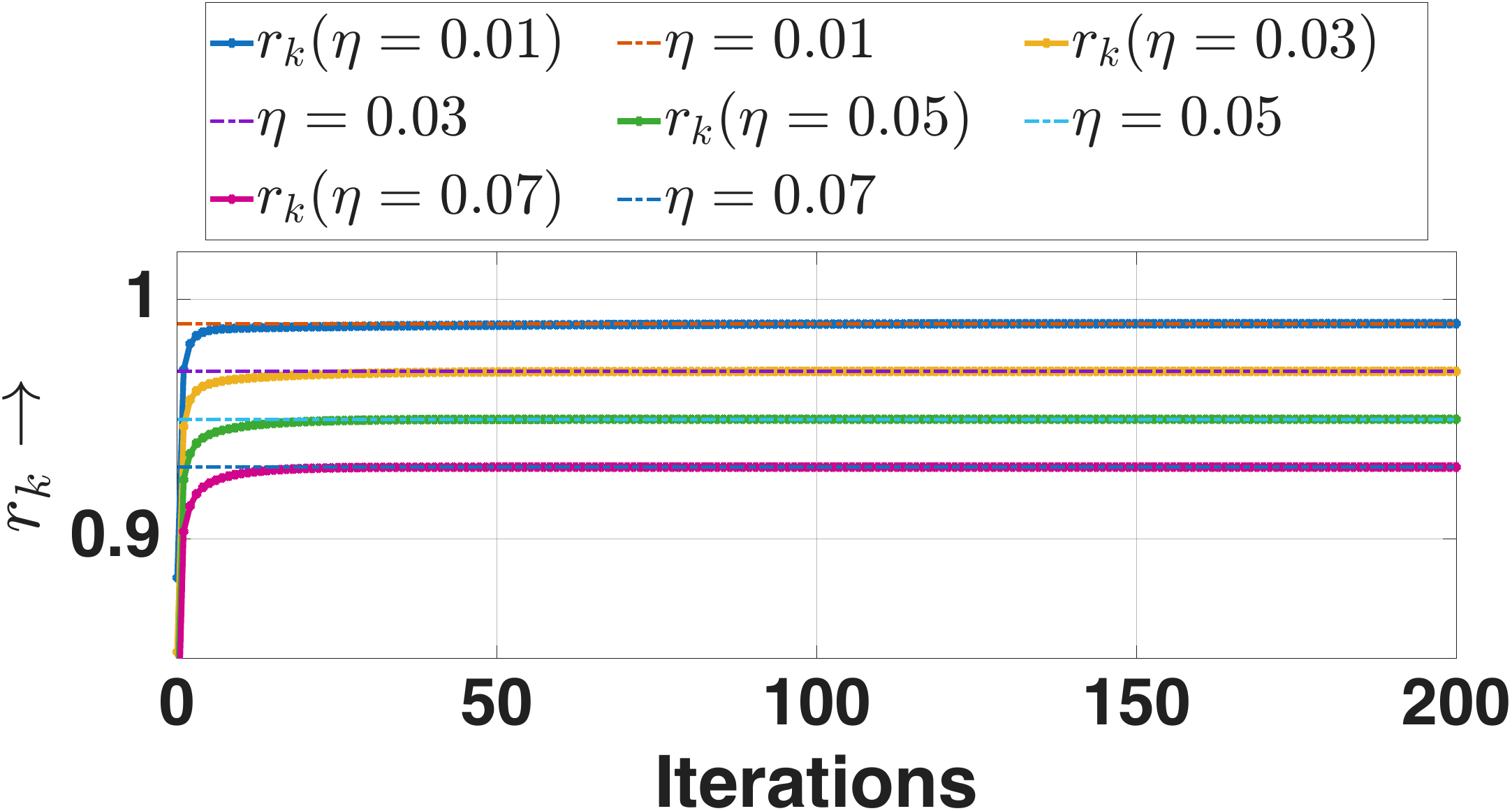}
            \label{subfig:p02_er}
        \end{subfigure}\vspace{1.5mm}
        \begin{subfigure}{\columnwidth}
            \includegraphics[scale = 0.18]{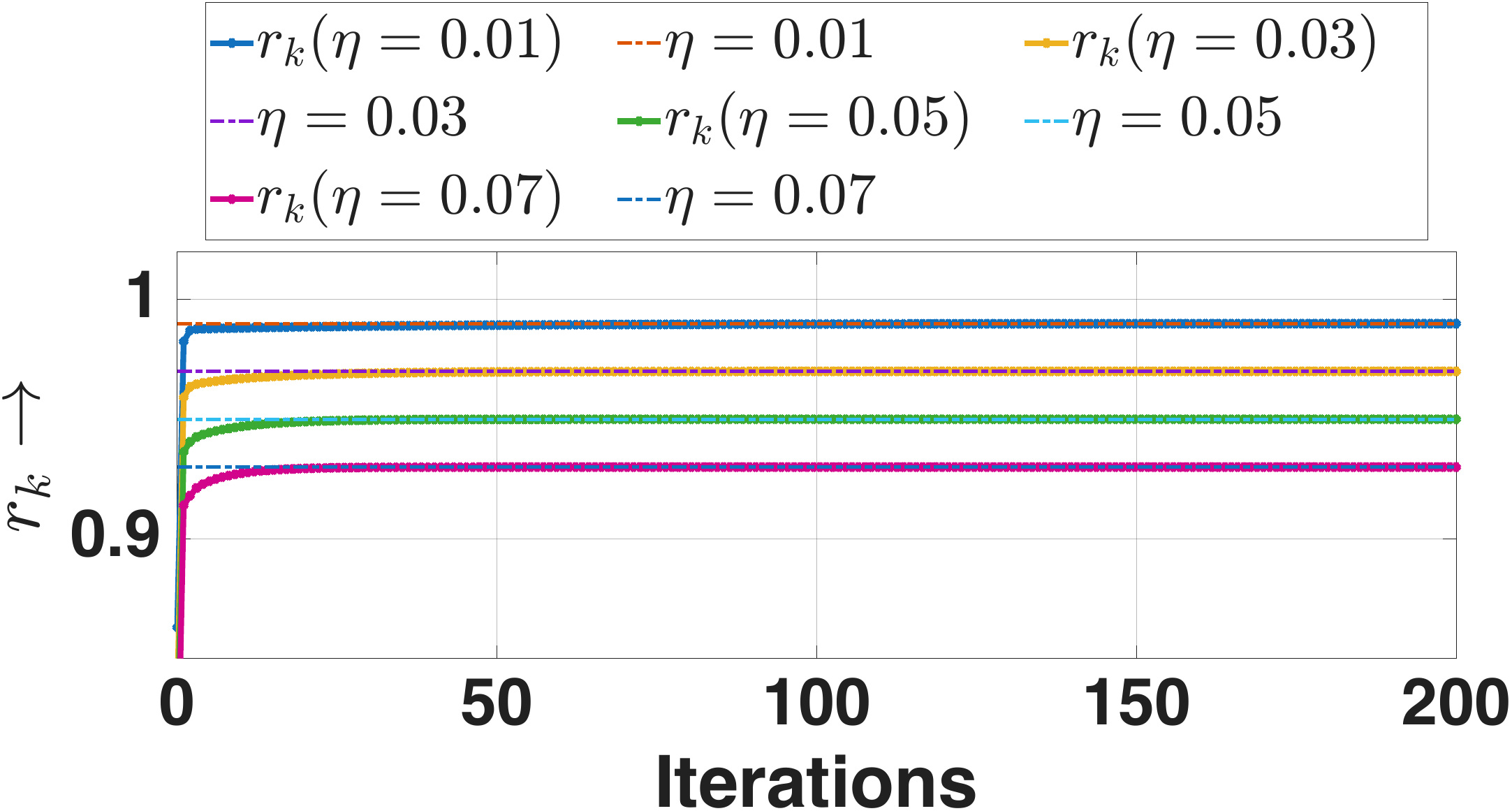}
            \label{subfig:p075_er}
        \end{subfigure}
        \label{subfig:row2_er}
    \end{subfigure}
    \caption{Rate of convergence \eqref{eq:roc} for Erd\H{o}s–R\'enyi  graph with \(p = 0.07\) (top) and \(p = 0.28\) (bottom).}
    \label{fig:roc_er}
\end{figure}
Finally, we observed a slightly improved rate of convergence when the spectral norm of the graph is small, even for larger values of \(\stsz\). This trend is consistent with the conclusion in Theorem \ref{th:combine analysis}, for both the \(d\)-regular expander graph and the Erd\H{o}s-R\'enyi graph. An exact characterization of the region of ($\sigma, \stsz$) where this locally approximate asymptotic convergence rate of
 $1 -\stsz $ for the optimality gap holds, remains an open problem.

\subsubsection*{Comment about scalability of Network-GIANT} 
We report in Figure \ref{fig:Scalablity}  the scalability of Network-GIANT. Two scenarios are considered. In the first case, we fixed \(\sigma\). The connection probability \(p\) is tuned to keep the spectral radius \(\sigma = 0.7128\) at a constant value. It is observed in Figure \ref{subfig:varying p} that even with increasing \(N\) the number of iteration required to reach a value of tolerance \(\log \big(||f(x_{k}) - f(x^{*})|| \big) \approx 10^{-10}\), remains constant, ensuring that Network-GIANT has the same convergence rate across all \(N\). In the second case, we fixed \(p = 0.04\). We reported the number of iterations required by the consensus error, gradient tracking error, and the optimality error of both algorithms to reach a tolerance of  \(10^{-7}\). Increasing \(N\) decreases \(\sigma\) (since $p$ is greater than $\frac{1}{N}$, the threshold for appearance of a giant component) as observed in Table \ref{tab:varying sigma}, thereby increasing the graph connectivity. 
 \vspace{-2mm}
 \begin{table}[htbp]
\centering
\begin{adjustbox}{max width=\linewidth}
\begin{tblr}{
  colspec = {l|cccc},
  row{1} = {font=\bfseries},
  hlines,
}
\(N\) 
  & \(50\) 
  & \(100\) & \(150\)
  & \(200\)  \\

\(\sigma\)
  & $0.973914$  & \(0.939\) & \(0.9389\)
  & $0.9307$ \\

\end{tblr}
\end{adjustbox}
\caption{Values of \(\sigma\) for different \(N\).}
\label{tab:varying sigma}
\end{table}
\vspace{-2mm}
 This explains why in Figure \ref{subfig:fixed p} the number of iterations required by Network-GIANT to reach a tolerance of \(10^{-7}\) decreases as \(N\) increases. We picked \(\eta\) to be \(0.02\) and \(0.065\) for GradTrack and Network-GIANT, respectively.
 \begin{figure}[h]
    \centering
    \begin{subfigure}{0.48\columnwidth}
        \centering
        \includegraphics[scale = 0.12]{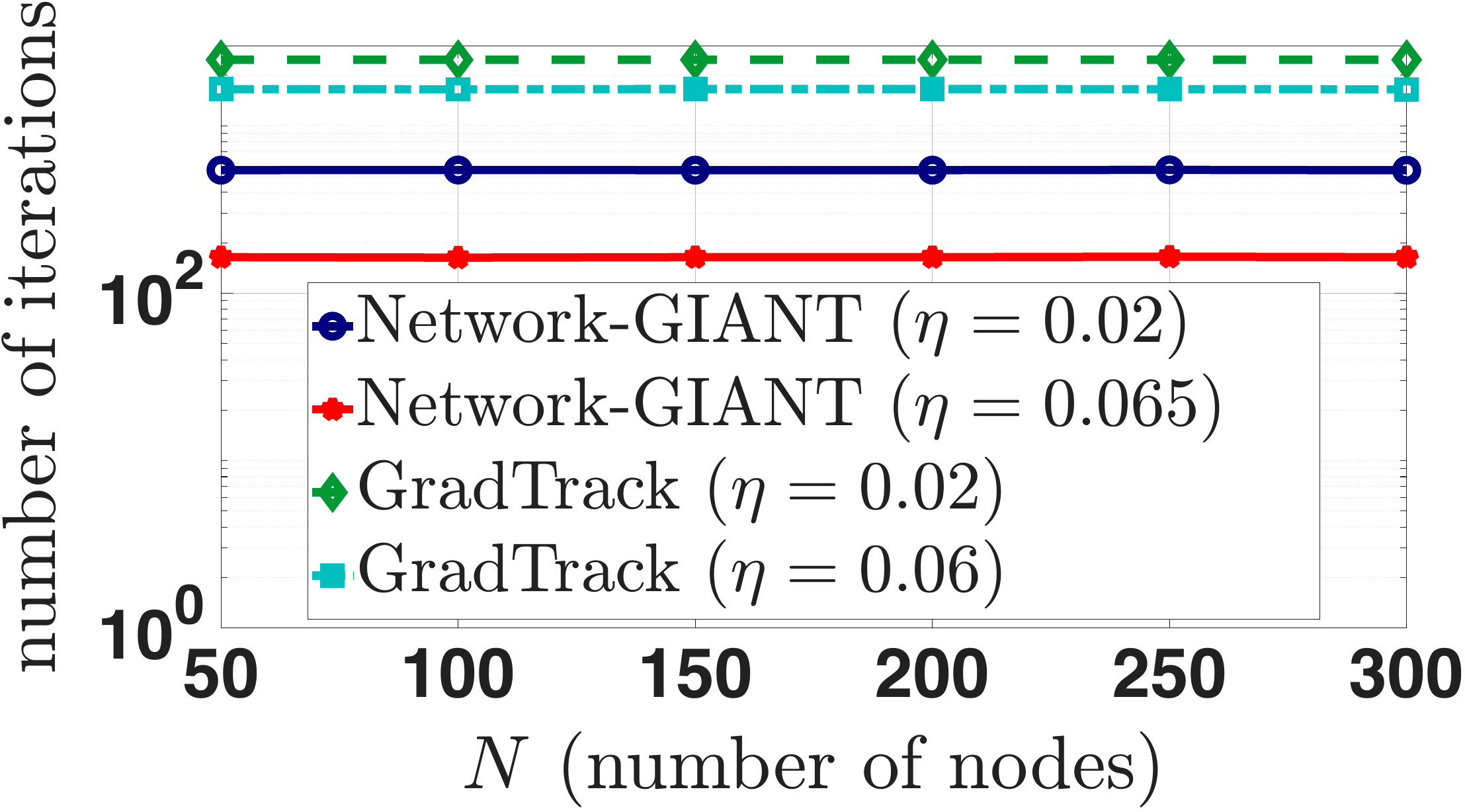}
        \caption{Varying \(p\)}
        \label{subfig:varying p}
    \end{subfigure}\hfill
    \begin{subfigure}{0.48\columnwidth}
        \centering
        \includegraphics[scale = 0.1011]{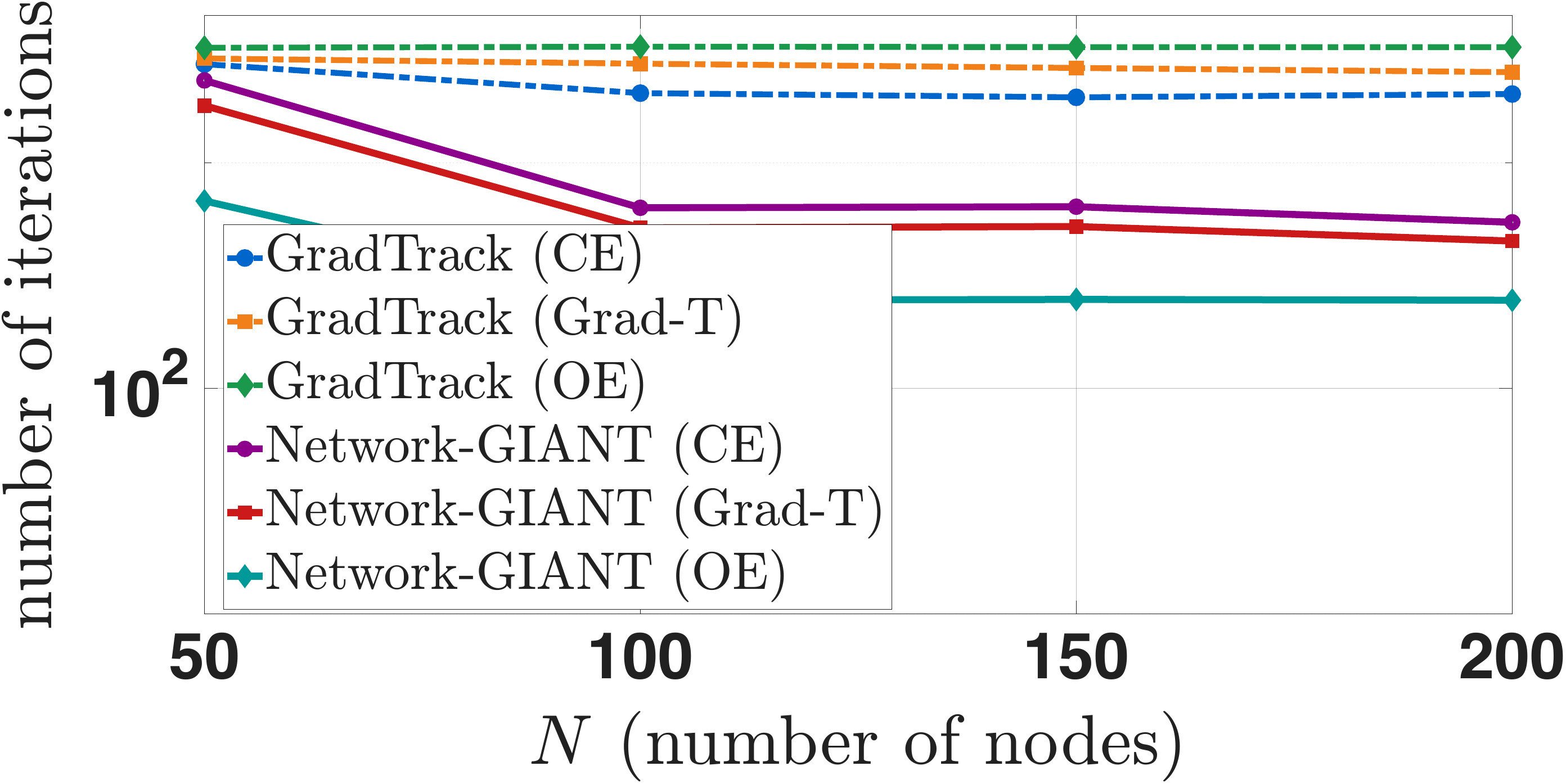}
        \caption{Fixed \(p = 0.04\)}
        \label{subfig:fixed p}
    \end{subfigure}
    \vspace{1mm}
    \caption{Scalability of Network-GIANT for Erd\H{o}s–R\'enyi graph. Here `CE', `Grad-T', and `OE' are short forms for consensus error, gradient-tracking error, and optimality error, respectively.}
    \label{fig:Scalablity}
\end{figure}
\subsubsection*{Comparison with state-of-the-art algorithms}
For comparing \(\ngiant\) with the gradient-tracking based first-order method developed in \cite{li_na_gradient_tracking}, termed as `GradTrack' for convenience, and a Nesterov-accelerated version of the gradient tracking based optimization algorithm 
(ACC-NGD-SC) \cite{ref:GQ-NL-19}, we considered the full CovType dataset, and solved the multinomial logistic classification problem \cite[Chapter 4]{ref-TH-RT-JF-09}. Specifically, we solve the distributed optimization problem \eqref{eq:problem_formulation/optimization}, where local objective function is defined by
\begin{equation*}
    \objfunc_i(\dvar) \Let - \frac{1}{m_i}\sum_{j=1}^{m_i} \sum_{l=1}^{M} \indic{\aset[]{v^i_j=l}} \log\Big(\varphi_l(u^i_j;\dvar) \Big) + \regu \norm{\dvar}^2,
\end{equation*}
\(\regu>0\) is the regularization penalty, and \begin{align*}
    \varphi_l(u^i_j;\dvar) = \begin{cases}
      \frac{\exp \big((\dvar^l)^{\top}u^i_j \big)}{1 + \sum_{l=1}^{M - 1} \exp\big( (\dvar^l)^{\top}u^i_j\big)} \, &\text{for }l=1,\ldots, M - 1,\\
      \frac{1}{1 + \sum_{l=1}^{M - 1} \exp\big( (\dvar^l)^{\top}u^i_j\big)}  \, &\text{for }l=M.
    \end{cases}
\end{align*}
Here the parameters \(\dvar \Let (\dvar^1 \, \dvar^2\, \cdots\, \dvar^{M- 1})^{\top}\) are \((M-1) \times (d+1) \)-dimensional matrices, where \(\dvar^l \in \Rbb^{N+1}\) for each \(l \in \aset[]{1, \ldots, M}\). The full CovType dataset consists of \(581012\) data points with \(54\)-dimensional feature vectors distributed among \(7\) classes. For these experiments, the regularizer \(\regu\) was chosen to be \(0.001\).  For the comparison we considered \(N = 50\). The empirically optimized stepsizes chosen for different algorithms are shown in Table \ref{tab:step-size_comparing}.
\begin{table}[htbp]
\centering
\begin{adjustbox}{max width=\linewidth}
\begin{tblr}{
  colspec = {l|cc|cc},
  row{1} = {font=\bfseries},
  hlines,
}
Algorithms 
  & \(d=4\)   
  & $d = 14$ 
  & \(p = 0.2\)  
  & \(p=0.75\) \\

GradTrack 
  & 0.08 & 0.18 
  & 0.08 & 0.12 \\

\(\ngiant\) 
  & 0.075 & 0.08 
  & 0.075 & 0.08 \\

ACC-NGD-SC 
  & 0.09  & 0.29 
  & 0.2 & 0.3 \\

\end{tblr}
\end{adjustbox}
\caption{Stepsize \(\stsz\) for different algorithms and graph models.}
\label{tab:step-size_comparing}
\end{table}
\begin{figure}[!htbp]
    \centering
    \begin{subfigure}{0.85\columnwidth}
        \includegraphics[scale = 0.13]{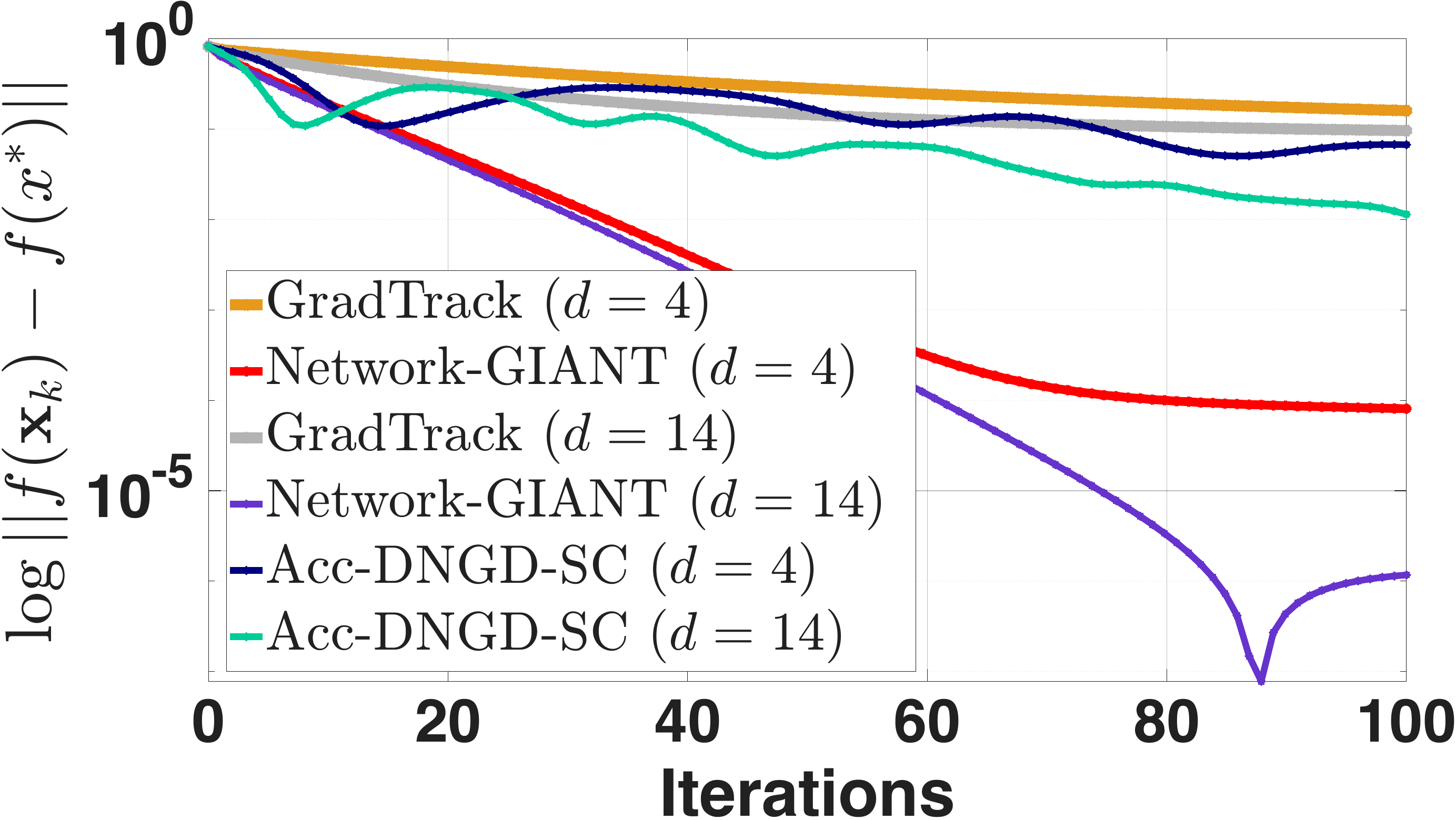}
        \label{subfig:compare_expanderr}
    \end{subfigure}\hfill
    
    \begin{subfigure}{0.85\columnwidth}
        \includegraphics[scale = 0.13]{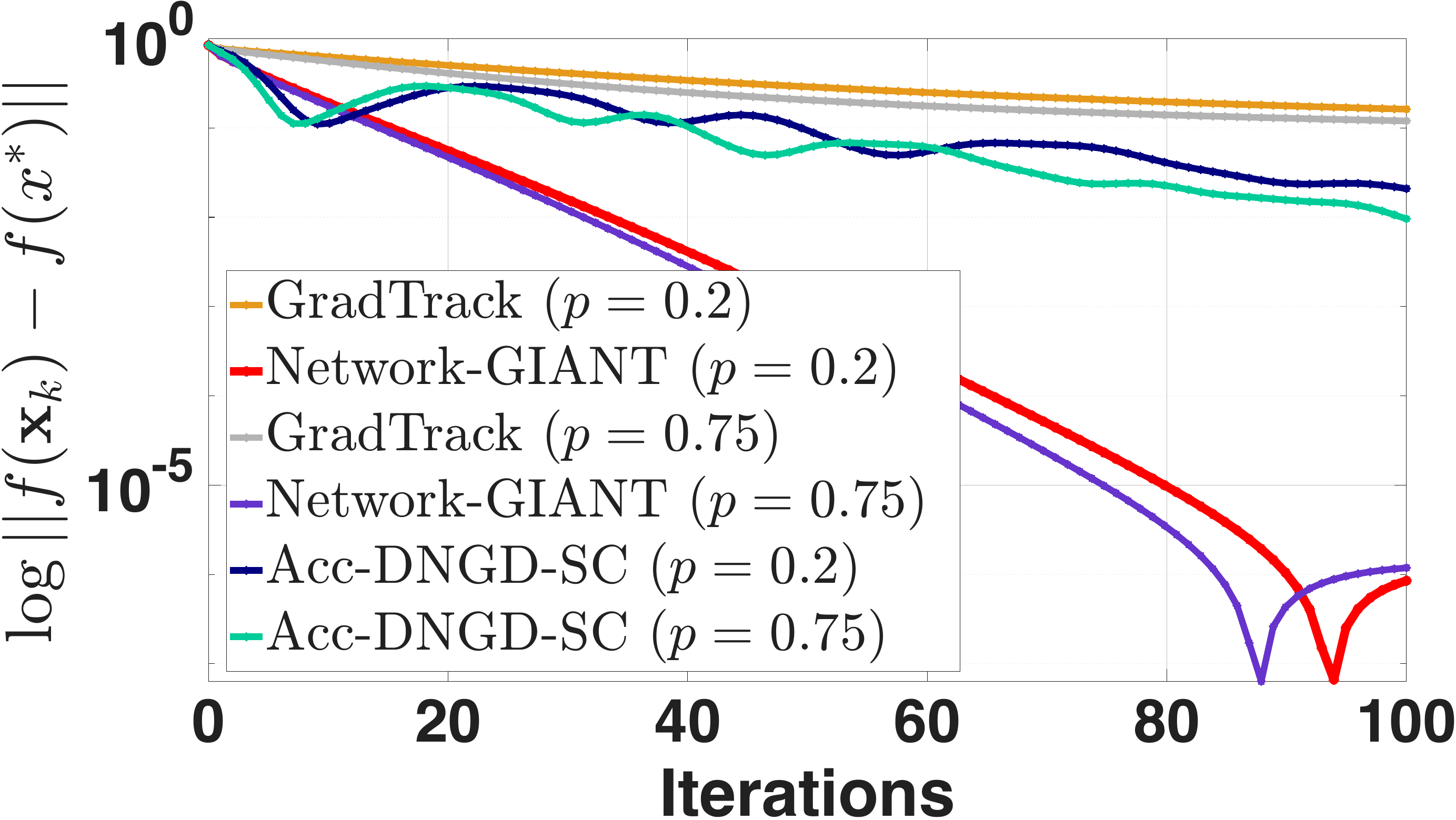}
        \label{subfig:compare_ERr}
    \end{subfigure}
    \caption{Comparison of \(\ngiant\) with GradTrack (\cite{li_na_gradient_tracking}) and ACC-NGD-SC (\cite{ref:GQ-NL-19}) for \(d\)-regular expander graph with \(d=4\) and \(d=14\) (top) and Erd\H{o}s-R\'enyi graph with \(p=0.2\) and \(p=0.75\) (bottom).}
    \label{fig:comparee}
\end{figure}

For ACC-NGD-SC \cite{ref:GQ-NL-19}, we considered a grid of \(\lcrc{0.035}{0.3}\), and the step-sizes, \(\stsz \in \lcrc{0.035}{0.3}\), corresponding the values reported in Table \ref{tab:step-size_comparing} corresponding to the best result obtained. As stated in \cite{ref:GQ-NL-19}, the momentum parameter, denoted here by \(\alpha\) was chosen to be \( \sqrt{\mu \stsz}\) , where \(\mu\) is the strong convexity parameter. For the dataset considered, \(\mu \simeq 0.002\) was computed empirically.\footnote{In our experiments, we observed that a larger \(\alpha\) can be appropriately chosen to obtain a damped response; see Figure \ref{fig:compareer} for more details.} 

It is observed that Network-GIANT comfortably outperforms the ACC-NGD-SC algorithm as well, illustrating that using approximate curvature information can yield significant improvements in convergence rates in distributed optimization. For a densely connected graph, the convergence is more prominent as expected. \footnote{
The dip observed in Figures \ref{fig:comparee} and \ref{fig:compareer} during the early stages of training may be caused by inaccuracies in the Hessian approximation. Since approximate Newton methods rely on estimated curvature information, the resulting descent direction may be overly aggressive, causing the iterates to overshoot the optimal trajectory and leading to a temporary \emph{spike} or \emph{dip} in the empirical loss function.}
\begin{figure}[!htbp]
    \centering

    \begin{subfigure}{\columnwidth}
        \centering
        \begin{subfigure}{0.85\columnwidth}
            \includegraphics[scale = 0.13]{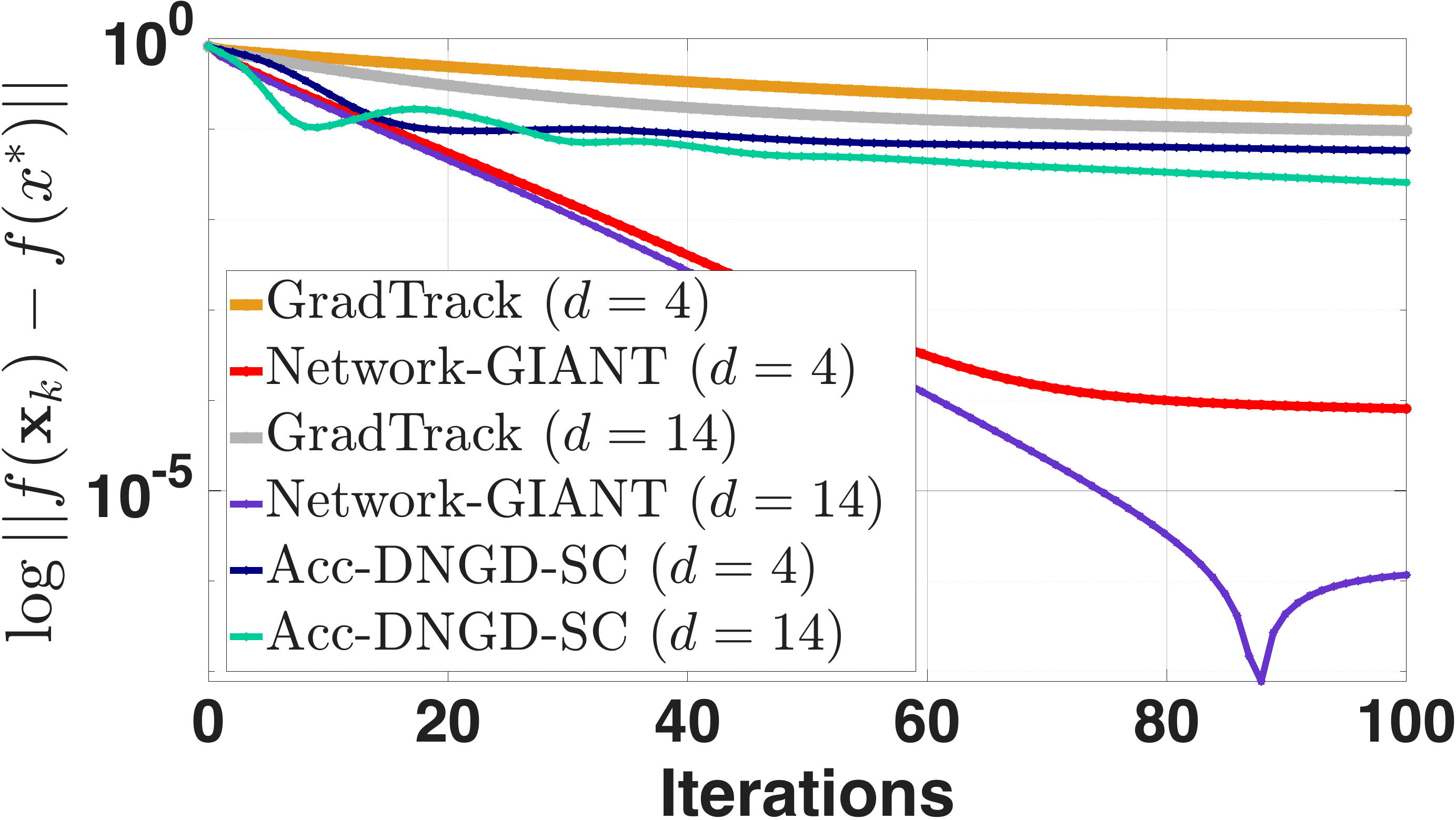}
            \label{subfig:compare_ex_modd}
        \end{subfigure}\hfill
        \begin{subfigure}{0.85\columnwidth}
            \includegraphics[scale = 0.13]{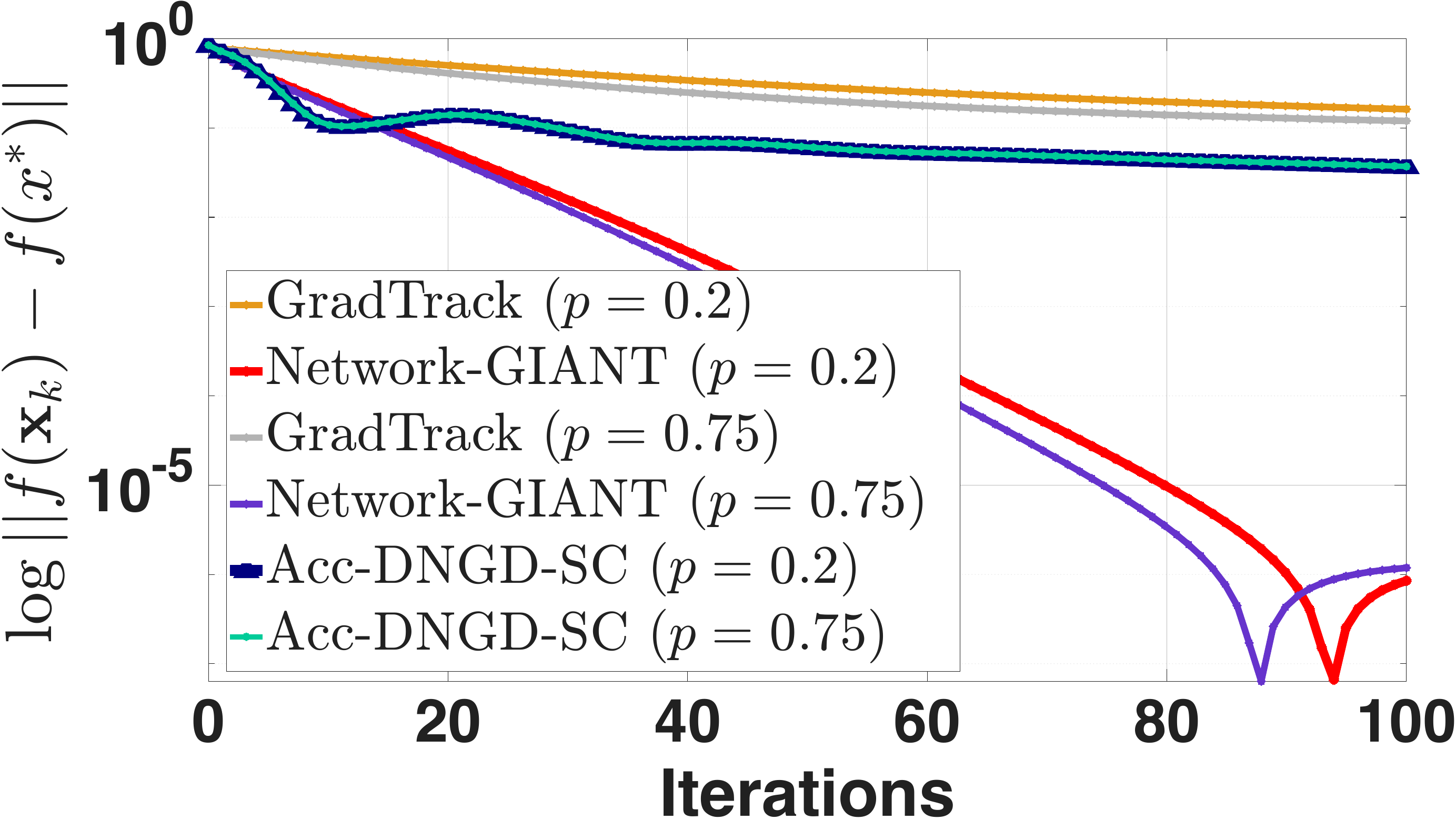}
            \label{subfig:compare_er_modd}
        \end{subfigure}
    \end{subfigure}
    \caption{Comparison when \(\alpha > \sqrt{\mu \eta}\) was chosen, keeping all other parameters identical, and considered \(\alpha = 0.05 \). A damped response for ACC-NGD-SC \cite{ref:GQ-NL-19} is reported.}
    \label{fig:compareer}
\end{figure}
Figure \ref{fig:wallclock} provides a comparison of Network-GIANT with Grad-Track \cite{li_na_gradient_tracking} and the accelerated first-order algorithm Acc-DNGD-SC \cite{ref:GQ-NL-19}, in terms of its wall-clock time, for a  \(d\)-regular expander graph. 
\begin{figure}[!htbp]
    \centering
    \begin{subfigure}{0.48\columnwidth}
        \centering
        \includegraphics[scale = 0.0965]{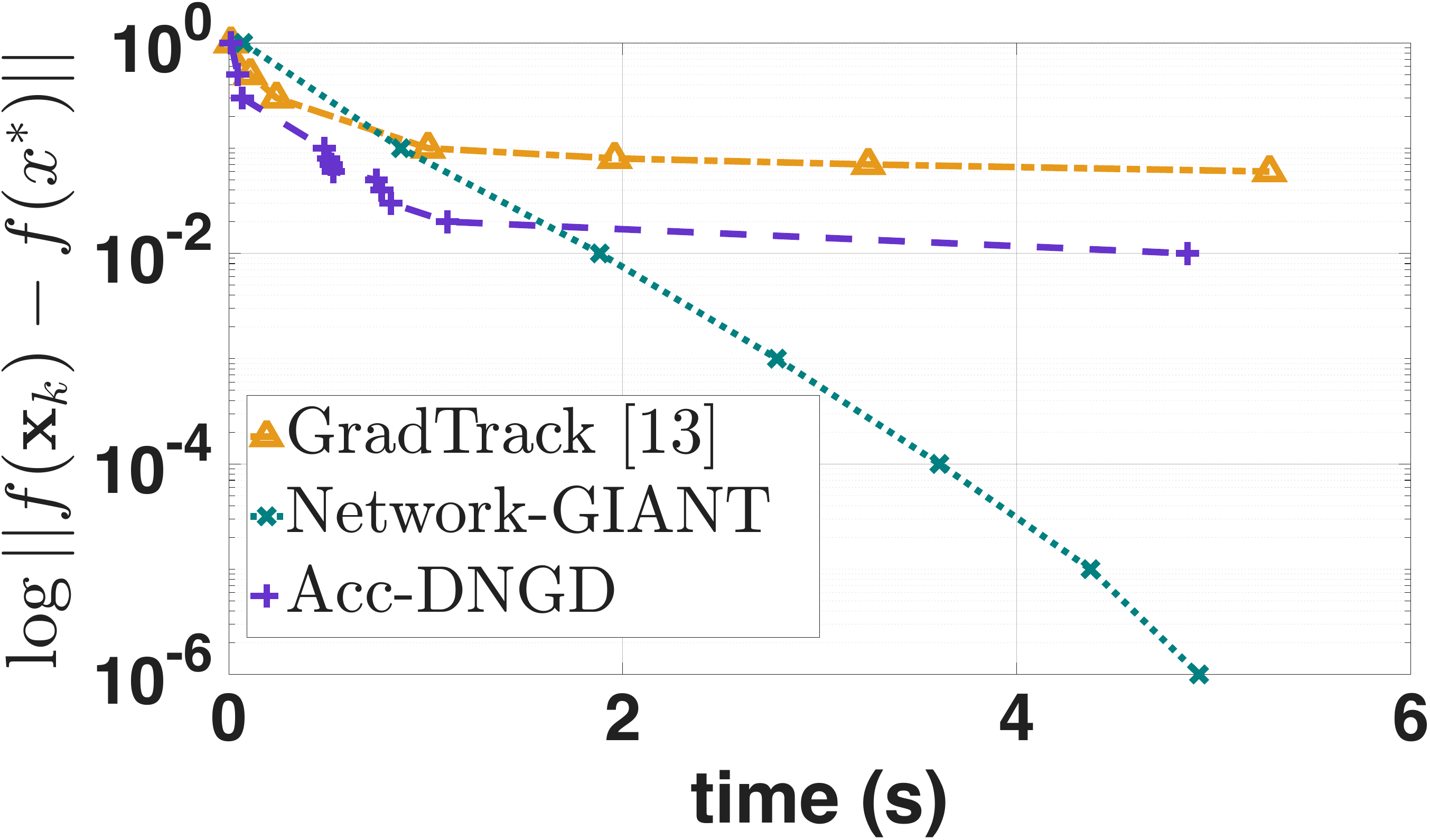}
        \caption{Degree equals to \(14\).}
        \label{subfig:degree14}
    \end{subfigure}\hfill
    \begin{subfigure}{0.48\columnwidth}
        \centering
        \includegraphics[scale = 0.0965]{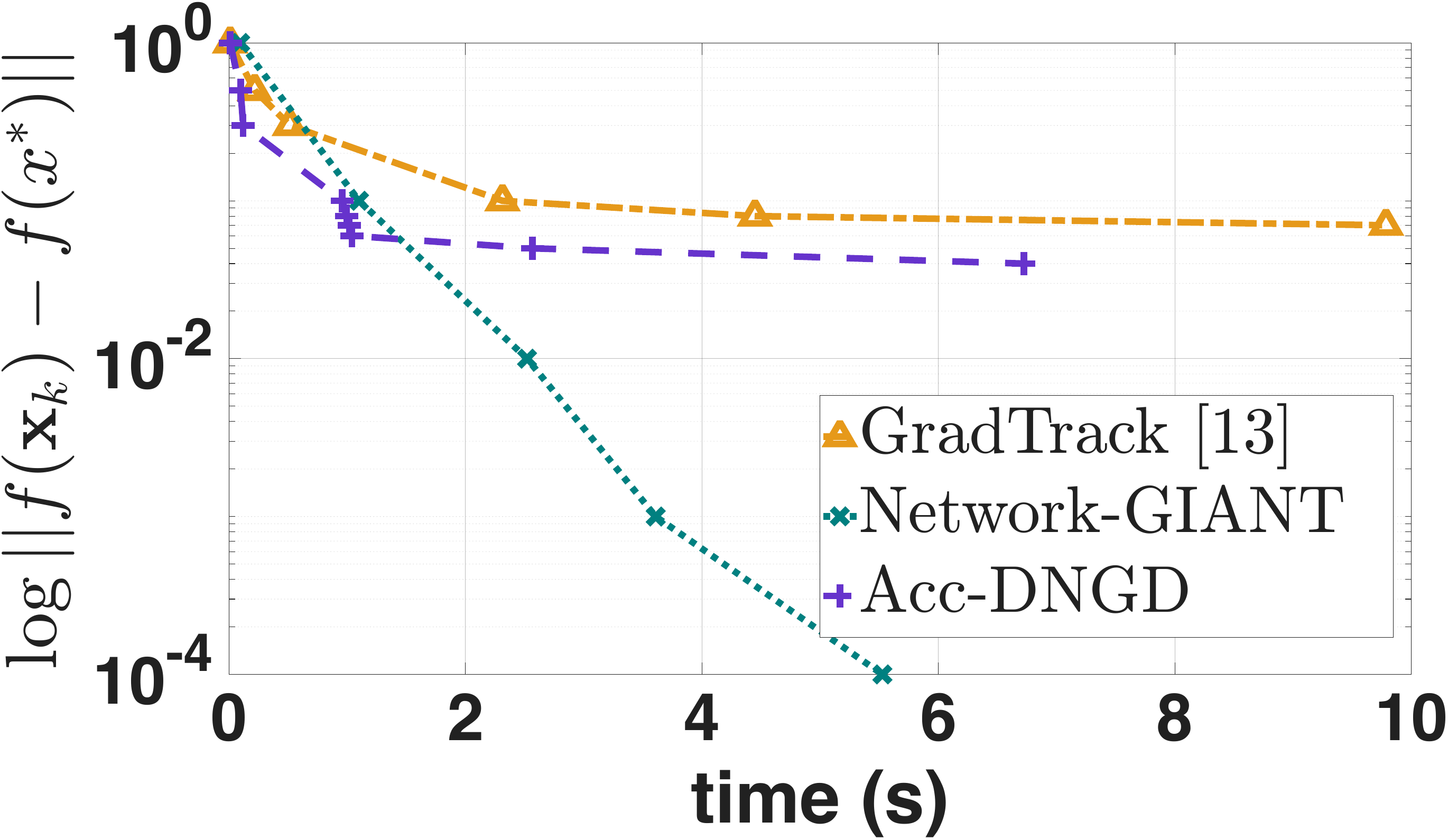}
        \caption{Degree equals to \(4\).}
        \label{subfig:degree4}
    \end{subfigure}
\vspace{1mm}
    \caption{Comparison with respect to the wall-clock time.}
    \label{fig:wallclock}
\end{figure}
This further confirms that the faster convergence rate achieved by the distributed Newton-type step offsets the computational bottleneck due to the local Hessian inverse. For high-dimensional problems, one may employ further approximation techniques, as discussed in Remark \ref{rem:further approx}, to reduce computational efforts, and recent advances in computational hardware are expected to make such implementations more practical in the near future.

\subsubsection*{Performance under heterogeneous data distribution}
We introduce distribution-based label skew heterogeneity into the data, following the partition scheme detailed in \cite{ref:JZ-ZL-BL-JX-SW-SD-CW-22}, where each agent is allocated a proportion of the samples of each label/class according to a Dirichlet distribution with sampling probability \(p_k \sim \mathrm{Dir}(\hetero)\) for each class \(k\). A portion \(p_{k,j}\) of the sampled data from class \(k\) is then allocated to agent \(j\). Here, the parameter \(\hetero\)  controls the degree of skewness, and typically a smaller value \(\hetero<0.1\) denotes highly skewed data distributions, whereas a high value of \(\hetero\) represents a homogeneous data distribution.

We compare Network-GIANT under varying levels of data heterogeneity with Grad-Track \cite{li_na_gradient_tracking} and with Network-GIANT under homogeneous data distributions on d-regular expander graphs of different connectivity; see Figure \ref{fig:hetero_ex} for details. The legend entries without a corresponding \(\hetero\) value represent the error curves obtained under homogeneous data distributions. For Grad-Track, we report only the results for the homogeneous setting, as its behavior does not differ significantly under heterogeneous data distributions.
\begin{figure}[!htbp]
    \centering
    \begin{subfigure}{0.48\columnwidth}
        \centering
        \includegraphics[scale = 0.096]{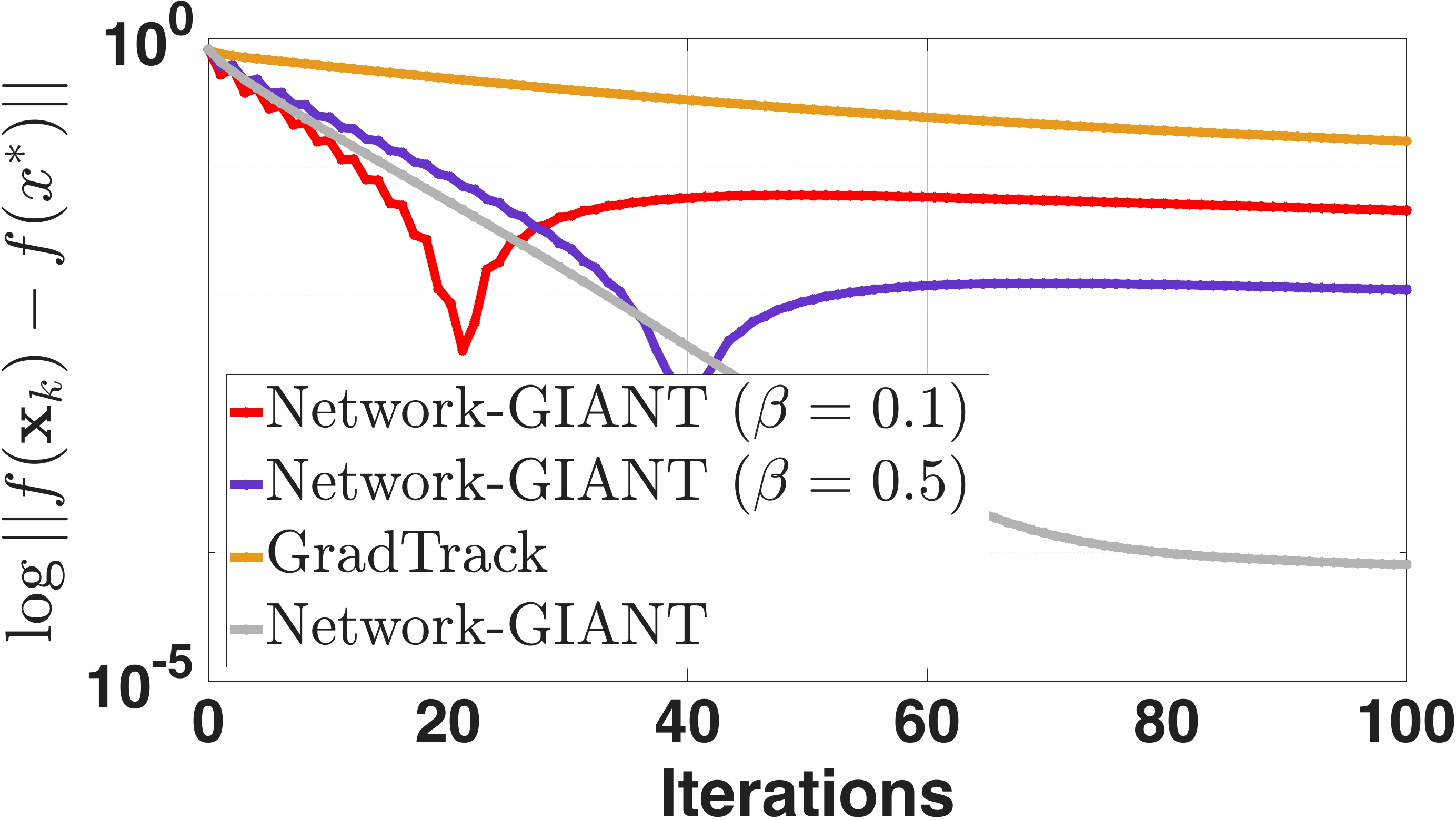}
        \caption{ \(d=4\)}
        \label{subfig:ex_ber_}
    \end{subfigure}\hfill
    \begin{subfigure}{0.48\columnwidth}
        \centering
        \includegraphics[scale = 0.096]{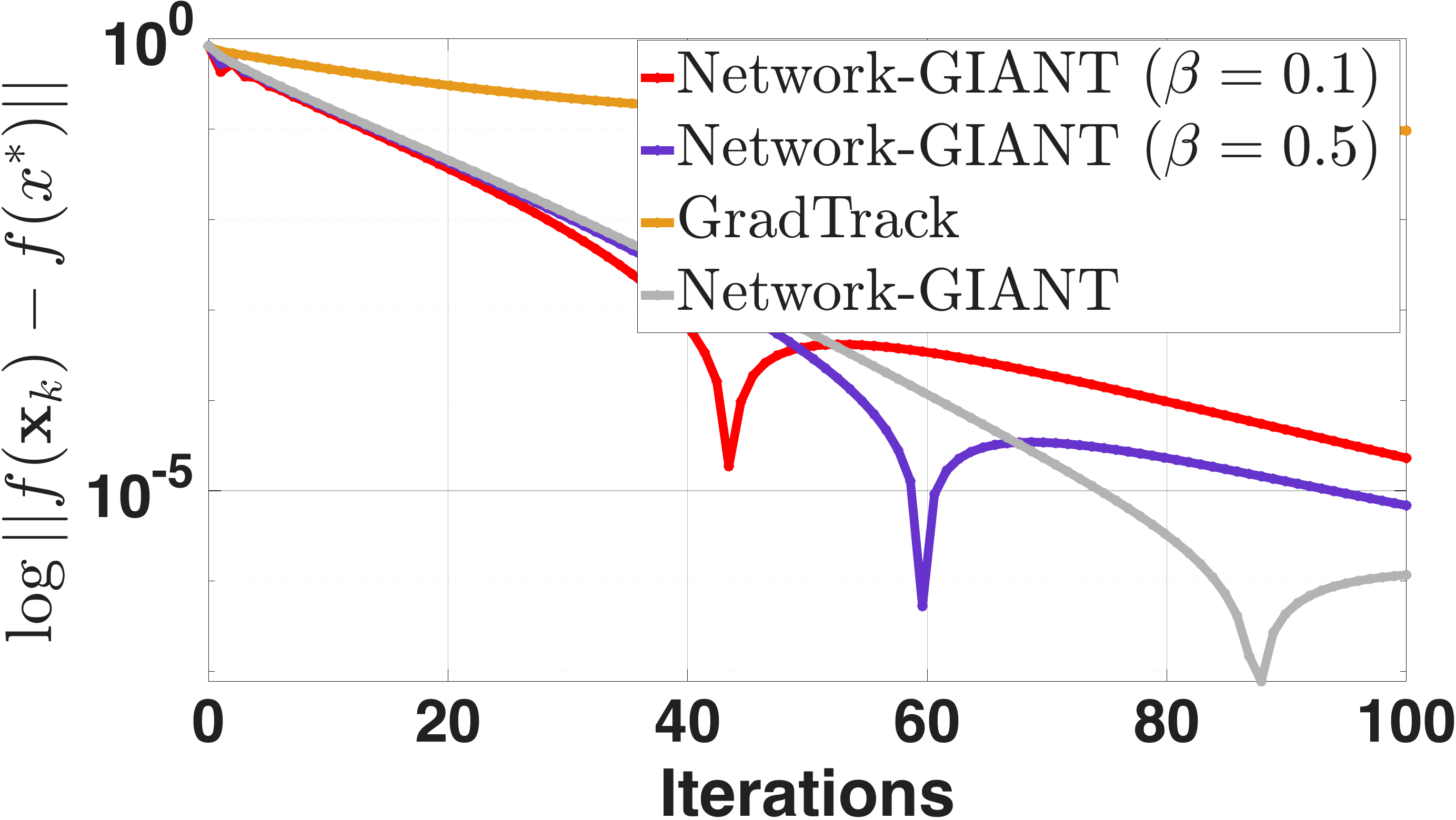}
        \caption{\(d=14\)}
    \end{subfigure}
    \vspace{1mm}
    \caption{Performance under heterogeneous data distribution.}
    \label{fig:hetero_ex}
\end{figure}

As observed in Figure \ref{fig:hetero_hess_error}, heterogeneity introduces bias that degrades the Hessian approximation. While densely connected graphs exhibit some robustness to heterogeneous data distributions, the effect is considerably more pronounced in sparsely connected graphs; see Figures \ref{subfig:ex_ber_}.
\begin{figure}[!htbp]
    \centering
        \includegraphics[scale = 0.125]{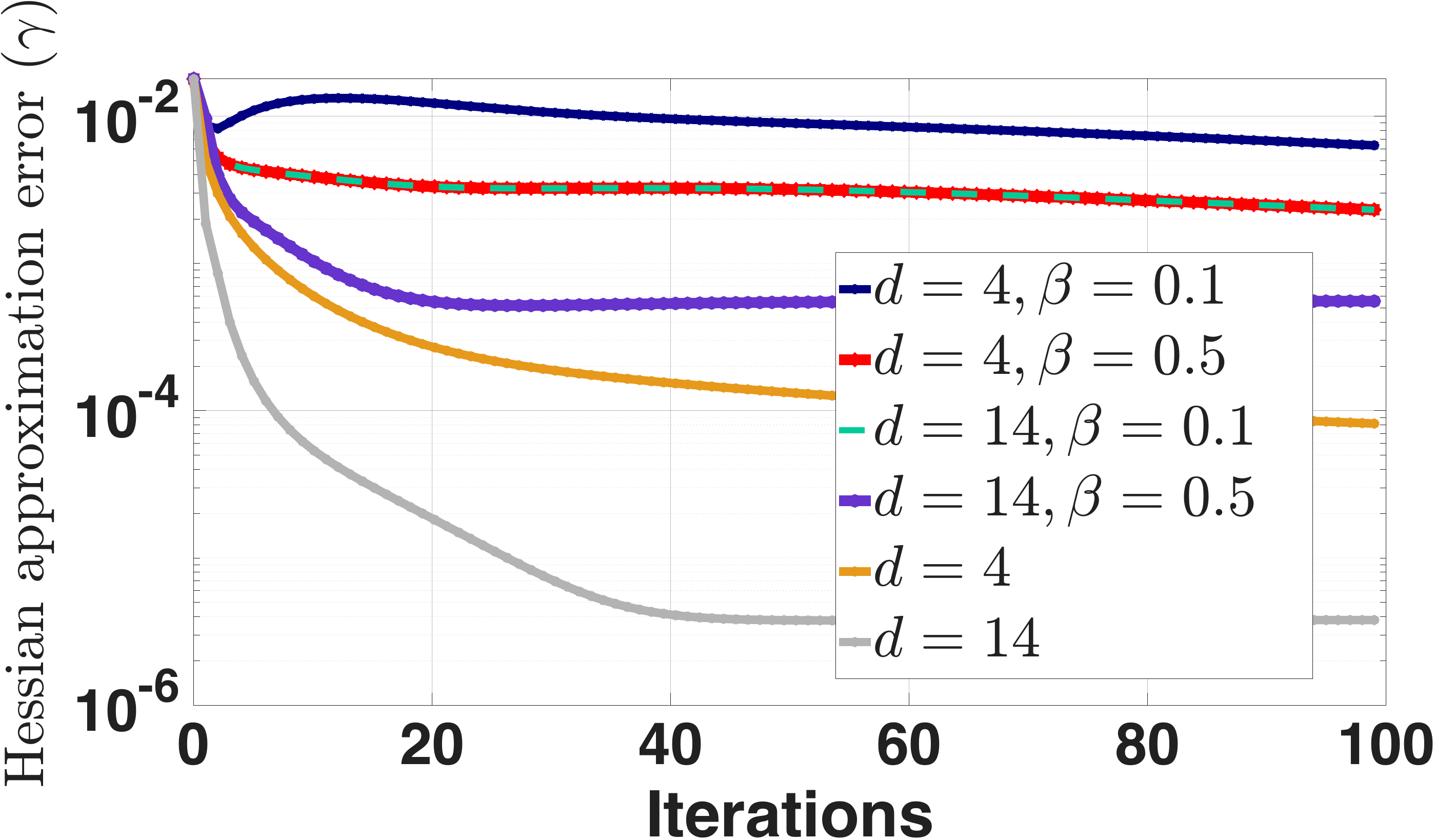}
    \caption{Hessian approximation error for \(d\)-regular expander graph.}
    \label{fig:hetero_hess_error}
\end{figure}

However, we emphasize that Theorems \ref{lin_convergence} and \ref{th:linear rate} assert global linear convergence irrespective of whether the data distribution among agents is homogeneous or heterogeneous. As it can be seen, by choosing a sufficiently small stepsize \(\eta= 0.05\) for the \(4\)-regular expander graph with \(\hetero= 0.1\), Network-GIANT still converges faster than its first-order counterparts, as clearly evidenced in Figure \ref{subfig:ex_ber_}.

	\section{Conclusions}
	This paper focused on the derivation of a number of detailed convergence results of a recently developed approximate Newton-type fully distributed optimization algorithm called Network-GIANT. For strongly convex local cost functions with the step size below an explicit bound, a global linear convergence result is shown, which can be explicitly computed as the spectral radius of a $3\times 3$ matrix, the elements of which depend on the cost function Lipschitz continuity ($L$) and strong convexity ($\mu$)  parameters, the step size ($\eta$) and the spectral norm ($\sigma$) of the underlying consensus matrix. A mixed linear-quadratic inequality-based iterative upper bound is also established for the optimality gap, which indicates that for small step sizes and a negligible Hessian approximation error, Network-GIANT achieves an asymptotic approximate local linear convergence rate of $1-\eta$, which provides a mathematical justification as to why Network-GIANT is empirically seen to achieve a faster convergence rate than its first order gradient-tracking based counterparts. Ongoing and future work will focus on extensions of Network-GIANT to graphs with compressed and noisy communication, incorporating acceleration techniques such as Nesterov-acceleration and heavy-ball type methods, and more general (e.g., directed) graphs. 
    
    \begin{appendices}
        \section{Proof of Theorem 1}\label{appen:th_1}
        We prove three norm inequalities in \eqref{norminequality} individually. 
        \par \noindent
        {\bf Consensus error norm}: 
        First, note that, using the parameter update equation  of 
        \eqref{netgalg_global}, we have \(\mathbf{x}_{k+1} - \mathbf{1} \bar x_{k+1} = W \mathbf{x}_{k} - \mathbf{1} \bar x_{k} - 
        \eta (\mathbb{I} - \frac{1}{N} \mathbf{1}\mathbf{1}^T) \mathbf{y}_k,\) using $W \mathbf{1} =\mathbf{1}$ 
        where 
        $\mathbb{I}$ denotes the identity matrix of order $N \times N$.
        Noting that the 2-norm  $\norm{\mathbb{I}  - \frac{1}{N} \mathbf{1}\mathbf{1}^T}_2 \leq 1$ for all $n > 1$,  $\norm{\mathbf{y}_k} \leq 
        \frac{1}{\mu}\norm{\mathbf{s}_k} $, and $\norm{W \mathbf{x}_{k} - \mathbf{1} \bar x_{k}} \leq \sigma \norm{\mathbf{x}_{k} - \mathbf{1} \bar x_{k}}$, it follows that $ \norm { \mathbf{x}_{k+1} - \mathbf{1} \bar x_{k+1}} \leq  \sigma \norm{\mathbf{x}_{k} - \mathbf{1} \bar x_{k}} + \frac{\eta}{\mu} \norm{\mathbf{s}_k}.$ Finally using equation (ii) of Lemma \ref{lemmaslina}, we have 
\begin{align*}
    & \norm{\mathbf{x}_{k+1} - \mathbf{1} \bar x_{k+1}} \leq (\sigma + \eta \frac{L}{\mu})\norm{\mathbf{x}_{k} - \mathbf{1} \bar x_{k}} 
    \nonumber \\ 
    & + \frac{\eta}{\mu} \norm {\mathbf{s}_{k} - \mathbf{1} g_k} + \eta \frac{L}{\mu} \sqrt{N}  \norm{\bar {x}_{k} - x^{\star}} 
\end{align*}
\par \noindent
{\bf Gradient tracking error norm}: 
        Next, using a result from \cite{li_na_gradient_tracking}, we have \(\norm{\mathbf{s}_{k+1} - \mathbf{1} g_{k+1}} \leq \sigma \norm {\mathbf{s}_{k} - \mathbf{1} g_{k}} + L \norm{\mathbf{x}_{k+1} - 
        \mathbf{x}_{k}}.\)
        Next, by using the identity \(\mathbf{x}_{k+1} -  \mathbf{x}_{k} = (W-\mathbb{I}) \mathbf{x}_{k} - \eta  \mathbf{y}_k 
        = (W-\mathbb{I}) (\mathbf{x}_{k} - \mathbf{1} \bar x_{k}) - \eta \mathbf{y}_k,\)
        and $\norm{W-\mathbb{I}}_2 \leq 2$, we get (after substituting for 
        $\norm{\mathbf{y}_k} \leq \frac{1}{\mu} \norm{\mathbf{s}_k}$, and some rearranging) 
        \begin{align*}
           &  \norm{\mathbf{s}_{k+1} - \mathbf{1} g_{k+1}} \leq (2L + \eta \frac{L^2}{\mu}) \norm{\mathbf{x}_{k} - \mathbf{1} \bar x_{k}} 
            \\
           & + (\sigma+ \eta \frac{L}{\mu})  \norm{\mathbf{s}_{k} - \mathbf{1} g_{k}} + \eta\frac{L^2}{\mu} \sqrt{N}  \norm{\bar {x}_{k} - x^{\star}} 
        \end{align*}
         \par \noindent
        {\bf Optimality gap norm}: 
        To simplify the following analysis, we use $\tilde x_{k} =  \bar x^T_k, \tilde x^{\star} = (x^{\star})^T$, and $\tilde s_i^k = (s_i^k)^T$. 
        Once again, averaging both sides of the state/parameter update equation in \eqref{netgalg_global}, and using the shorthand notation $\tilde H_i^k = \nabla^2 f_i(x_i^k)$ for the local Hessians (note that these matrices are positive definite symmetric), we can write 
        \begin{align*}
         & \tilde{x}_{k+1} - \tilde x^{\star} = \tilde{x}_k - \tilde x^{\star} - \eta \frac{1}{N} \sum_{i=1}^N (\tilde H_i^k)^{-1} \tilde s_i^k \nonumber \\
         & = \tilde{x}_k - \tilde x^{\star} - \eta \frac{1}{N} \sum_{i=1}^N (\tilde H_i^k)^{-1} (\tilde s_i^k - \nabla f(\tilde{x}_k) + \nabla 
         f(\tilde{x}_k)) \end{align*}
         where $\nabla f(\tilde{x}_k) = (\nabla f(\bar{x}_k))^T$.
        Define 
        \begin{align*}   
         & \bar{H}_i^k = (\tilde H_i^k)^{-1} \nabla f(\tilde{x}_k)  \\
         &  = (\tilde H_i^k)^{-1} \int_0^1 \nabla^2 f( \tilde x^{\star} + \lambda(
        \tilde{x}_k - \tilde x^{\star}) ) d \lambda (\tilde{x}_k - \tilde x^{\star})
        \end{align*}
        Consequently, one can write \(\tilde{x}_{k+1} - \tilde x^{\star} = (\mathbb{I} - \eta \frac{\sum \bar{H}_i^k}{N})(\tilde{x}_k - \tilde x^{\star}) -\eta \frac{1}{N} \sum_{i=1}^N (\tilde H_i^k)^{-1} (\tilde s_i^k - \nabla f(\tilde {x}_k))\)
        Using a similar technique as in the Proof of Lemma 2, one can show that $\frac{\mu}{L} \mathbb{I} \leq \bar{H}_i^k \leq \frac{L}{\mu} 
        \mathbb{I}$, and hence 
        $ \left( 1 - \eta \frac{L}{\mu} \right) \mathbb{I} \leq  \mathbb{I} - \eta \frac{\sum \bar{H}_i^k}{N} \leq 
        \left(1 - \eta \frac{\mu}{L} \right)  \mathbb{I} $
        It follows then that if $\eta \leq \frac{\mu}{L}$, we have 
        $ 0 \leq \mathbb{I} - \eta \frac{\sum \bar{H}_i^k}{N}  \leq  \left(1 - \eta \frac{\mu}{L} \right)  \mathbb{I}, $ and finally, $\norm{\mathbb{I} - \eta \frac{\sum \bar{H}_i^k}{N} }_2 \leq (1 - \eta \frac{\mu}{L})$. This allows us to write \(\norm{\tilde{x}_{k+1} - \tilde x^{\star}} \leq \big(1 - \eta \frac{\mu}{L}\big) \norm{\tilde {x}_k - \tilde x^{\star}} + \eta \norm{\frac{1}{N} \sum_{i=1}^N (\tilde H_i^k)^{-1} (\tilde s_i^k - \nabla f(\tilde{x}_k))}  \).

        Denoting $S_a = \eta \norm{\frac{1}{N} \sum_{i=1}^N (\tilde H_i^k)^{-1} (\tilde s_i^k - \nabla f(\tilde{x}_k))}$, using the strong convexity of the local cost functions and the fact that the norm of the transpose of a vector is the same as the norm of the vector, it is easy to show that 
        \begin{align*}
         & S_a \leq \frac{\eta}{\mu} \frac{1}{N} \sum_{i=1}^N \norm{s_i^k - \nabla f(\bar{x}_k))} \\
           & \leq \frac{\eta}{\mu} \frac{1}{N} \sum_{i=1}^N \norm{s_i^k - \nabla f(\mathbf{x}_k) + \nabla f(\mathbf{x}_k)- \nabla f(\bar{x}_k))}  \\
           & \leq \frac{\eta}{\mu} \left( \frac{1}{N} \sum_{i=1}^N \norm{s_i^k - \nabla f(\mathbf{x}_k)} + 
           \norm{\nabla f(\mathbf{x}_k)- \nabla f(\bar{x}_k))} \right) \\
           & \leq \frac{\eta}{\mu} \frac{1}{N} \sum_{i=1}^N \norm{s_i^k - g_k} + \frac{\eta}{\mu} \frac{L}{\sqrt{N}} \norm{\mathbf{x}_k - \mathbf{1} \bar{x}_k}     
        \end{align*}
        where we have used the Lipschitz continuity of the gradients.

        Finally, noting again that $\norm{\tilde{x}_k - \tilde x^{\star}} = \norm{\bar{x}_k - x^{\star}}$, and 
        $$ \frac{1}{N} \sum_{i=1}^N \norm{s_i^k - g_k} \leq \sqrt{ \frac{1}{N} \sum_{i=1}^N {\norm{s_i^k - g_k}}^2 } = 
        \frac{\norm{\mathbf{s}_k - \mathbf{1} g_k}}{\sqrt{N}},$$
        we have \(\sqrt{N} \norm{\bar{x}_{k+1} - x^{\star}}  \leq (1 - \eta \frac{\mu}{L}) \sqrt{N} \norm{\bar{x}_k - x^{\star}} + \frac{\eta L}{\mu} \norm {\mathbf{x}_k - \mathbf{1} \bar{x}_k} + \frac{\eta}{\mu} \norm{\mathbf{s}_k - \mathbf{1} g_k}\).
        Combining the three inequalities for the three error norms, we obtain the inequality in \eqref{norminequality}. 
        \section{Proof of Theorem 2}\label{appen:th_2}
        To analyze the spectral radius of the matrix $G(\eta)$ defined in \eqref{norminequality}, we first recall that since $G(\eta)$ is a positive matrix, its spectral radius is real-valued and positive. Therefore, we need to investigate under what conditions on the step-size $\eta$, this real eigenvalue lies between $0$ and $1$. First, note that at $\eta=0$, the eigenvalues $G(0)$ are $\sigma, \sigma$, and $1$. Clearly, the spectral radius $\rho(G(\eta)) =1$ at $\eta=0$. 
        From continuity of eigenvalues as a function of the elements of the matrix, we know that $\rho(G(\eta))$ is a continuous function of $\eta$.  
        \par\noindent
        Next, we evaluate the derivative of $\rho(G(\eta))$ with respect to $\eta$ at $\eta=0$. Define $\alpha = 1-\eta \frac{\mu}{L}$, and $\beta = 
        \sigma + \eta \frac{L}{\mu}$, and $\delta = \eta \frac{L}{\mu}$ for simplifications. By analyzing the characteristic polynomial of $G(\eta)$, one can write that an eigenvalue 
        $\lambda$ of $G(\eta)$ satisfies the following equation \((\lambda-\alpha)(\lambda-\beta)^2 - (\lambda-\alpha)(2 \delta+ \delta^2)  - 2\delta^2 (\lambda-\beta)  - 2 \delta^3 - 2\delta^2 = 0\).
        Taking the derivative with respect to $\eta$, and substituting $\lambda=1$, and $\delta = 0$ at $\eta=0$, after some algebra one can get \(\big(\frac{\partial \lambda}{\partial \eta} + \frac{\mu}{L}\big) (1 - \sigma)^2 =0.\)
        Since $\sigma < 1$, this implies that $\frac{\partial \lambda}{\partial \eta} < 0$ at $\eta=0$. Therefore, from continuity of eigenvalues as a function of the matrix elements, it follows that the largest eigenvalue remains real and decreases as $\eta$ increases slightly from $0$.  Since the other two eigenvalues are $\sigma < 1$ at $\eta=0$, we obtain
        $\rho(G(\eta)) < 1$ for a sufficiently small $\eta > 0$. 
        \par \noindent
        Next, we check for what values of $\eta$, we have a solution $\lambda=1$. Clearly, one such value, as we saw earlier, is $\eta=0$. 
        By substituting $\lambda=1$ (and consequently, $(\lambda-\alpha) = \eta \frac{\mu}{L}$, $(\lambda-\beta) = 1 - \sigma -\eta \frac{L}{\mu}$) in the characteristic polynomial, and discarding the solution $\eta=0$, we get 
        (after some algebraic manipulations), surprisingly,  a single solution 
        $\bar \eta = \frac{(1-\sigma) \frac{\mu}{L}}{2(2-\sigma) (1 + (\frac{L}{\mu})^2})$. Since at $\eta =0$, the eigenvalues are $1,\sigma, \sigma$, and the spectral radius is less than $1$ in the immediate vicinity of $\eta >0$, and there is only one other value $\eta=\bar \eta$, where an eigenvalue can be $1$, we conclude that the magnitudes of all the eigenvalues (and hence the spectral radius) are less than $1$ for $0 < \eta < \bar \eta$. By noting that $\bar \eta < \frac{\mu}{L}$ (thus satisfying the global convergence condition on the step size for the damped Newton method), and substituting $\kappa = \frac{L}{\mu}$, we complete the proof of Theorem 2.

        \section{Proof of Theorem 3}\label{appen:th_3}
        Similar to the {\em optimality error norm} proof of Theorem 1, we work with the transposed notations $\tilde{x}_k, \tilde x^{\star}$, 
        $\tilde s_i^k$, and $\nabla f(\tilde{x}_k)$. We adapt a proof technique similar to that presented in \cite{JMLR_hazan}.
        As before, we can write 
        \begin{align*}
          &   \tilde{x}_{k+1} - \tilde x^{\star} \\
          & = \tilde{x}_k - \tilde x^{\star} - \eta \frac{1}{N} \sum_{i=1}^N (\tilde H_i^k)^{-1} (\tilde s_i^k - \nabla f(\tilde{x}_k) + \nabla 
         f(\tilde{x}_k)) \\
         &=  \tilde{x}_k - \tilde x^{\star}  - \underbrace{\eta \frac{1}{N} \sum_{i=1}^N (\tilde H_i^k)^{-1} (\tilde s_i^k - \nabla f(\tilde{x}_k))}_{S_1} \\
         & \;\;\;\;\;\;\;\;\;\;\;\;\;\;\;\;\;\;\;\;\;\;\;\;\;\;\;\;\;\; - \underbrace{\eta \frac{1}{N} \sum_{i=1}^N (\tilde H_i^k)^{-1}  \nabla f(\tilde{x}_k)) }_{S_2}
        \end{align*}
        We note that $S_2$ can be written as 
        $S_2 = \eta (H_{app}(\mathbf{x_k}))^{-1} \int_0^1 \nabla^2 f(\tilde x^{\star} + \lambda (\tilde{x}_k - \tilde x^{\star}) ) d \lambda (\tilde{x}_k - 
        \tilde x^{\star})$
        This allows us to write
        \begin{align*}
        &\tilde{x}_{k+1} - \tilde x^{\star}  = \Big( \mathbb{I} - \eta (H_{app}(\mathbf{x_k}))^{-1} \cdots \\& \hspace{1cm} \cdots \underbrace{\int_0^1\nabla^2 f(\tilde x^{\star} + \lambda (\tilde{x}_k - \tilde x^{\star}) ) d \lambda}_{\delta(\tilde x_k)} \Big) (\tilde{x}_k - 
        \tilde x^{\star}) + S_1    \\
        & = (1 - \eta) (\tilde{x}_k -\tilde x^{\star}) + \eta (H_{app}(\mathbf{x_k}))^{-1} \left[ H_{app}(\mathbf{x_k}) -\delta(\tilde x_k) \right] \\
        & \qquad \qquad \qquad \qquad \qquad\qquad\qquad\qquad \times (\tilde{x}_k -\tilde x^{\star}) + S_1
        \end{align*}
        It follows then from strong convexity and  $\eta < 1$
        \begin{align}
        & \norm{\tilde{x}_{k+1} - \tilde x^{\star}} \leq \left((1-\eta) +\frac{\eta}{\mu} \norm{H_{app}(\mathbf{x_k}) -\delta(\tilde x_k) } \right) 
          \nonumber \\
         & \qquad \qquad \qquad \qquad \times \norm{\tilde{x}_k -\tilde x^{\star}} + \norm{S_1} 
         \label{proof3eq1}
         \end{align}
         One can show  (recalling that $\norm{v^T} = \norm{v}$ for a vector $v$)
         \begin{align*}
         & \norm{H_{app}(\mathbf{x_k}) -\delta(\tilde x_k)} \\ 
         & \leq \norm{H_{app}(\mathbf{x_k}) - H_{tr}(\mathbf{x_k})}  
          + \norm{H_{tr}(\mathbf{x_k}) - \delta(\tilde x_k)} \\
         & \leq \gamma + \int_0^1 \norm{ \nabla^2 f(\mathbf{x}_k) - \nabla^2 f(\tilde x_k)} d \lambda \\
          & \qquad \qquad + \int_0^1 \norm{\nabla^2 f(\tilde x_k)  - \nabla^2 f(\tilde x^* + \lambda(\tilde x_k - \tilde x^*))} d \lambda \\
          & \leq \gamma + \frac{\bar L}{\sqrt{N}} \norm{\mathbf{x}_{k} - \mathbf{1} \bar x_{k}} + \frac{\bar L}{2} \norm{\bar x_k - x^{\star}}  
         \end{align*}
         where the last term in the last inequality is due to the Lipschitz continuity of the global Hessian. Using the above result in \eqref{proof3eq1}, we obtain (observing that since $\gamma < \mu$, and $\eta < 1$, $\eta(1 - \frac{\gamma}{\mu}) < 1$)
         \begin{align*}
          &  \norm{\bar{x}_{k+1} -  x^{\star}} \leq  \left(1- \eta(1 -\frac{\gamma}{\mu})\right) \norm{\bar {x}_{k} - x^{\star}} \\
        & + \frac{\eta \bar L}{\mu \sqrt{N}} \norm{\mathbf{x}_{k} - \mathbf{1} \bar x_{k}} \norm{\bar {x}_{k} - x^{\star}} 
         + \frac{\eta \bar L}{2 \mu} \norm{\bar {x}_{k} - x^{\star}}^2   + \norm{S_1} .
         \end{align*}
         Finally, noting that 
         $\norm{S_1}$ is the same as the quantity $S_a$ defined in the Proof of Theorem 1, we can obtain the final expression for the upper bound on $\norm{\bar{x}_{k+1} -  x^{\star}}$ given in \eqref{lin_quad_rate}. \\

        \vspace*{-0.5cm}
         \section{Proof of Theorem \ref{th:linear-quadratic rate}}\label{appen:proof_linear-quad-rate} 
         The proof of Theorem \ref{th:linear-quadratic rate} proceeds along the following steps:\\
         \noindent \textbf{Step 1: } 
         First, observe that for a sufficiently small \(\stsz\) picked according to Theorem \ref{lin_convergence}, the iterates corresponding to Network-GIANT converge at a linear rate.
         This asserts that there exist \(\constt_1, \constt_2>0\) and some \(\rrate_1, \rrate_2>0\) such that \(\norm{\mathbf{x}_{k} - \mathbf{1} \bar x_{k}} \leqslant \constt_1 e^{-\rrate_1 k}\) and \(\norm{\mathbf{s}_{k} - \mathbf{1} g_{k}} \le \constt_2 e^{-\rrate_2 k}\). Note that for \(\rrate_1, \rrate_2\) one can choose a sufficiently large \(\K \in \Nz\) such that \(\constt_1 e^{-\rrate_1 k} \le \ol{\constt}_1 \optgaperr_k^2 \) and \(\constt_2 e^{-\rrate_2 k} \le \ol{\constt}_2 \optgaperr_k^2 \) for each \(k \ge \ol{K}\), which implies \[\constt_2 e^{-\rrate_2 k} + L \constt_1 e^{-\rrate_1 k} \le \underbrace{(\ol{\constt}_2 + L \ol{\constt}_1)}_{\teL \ol{\constt}}\optgaperr_k^2 \text{ for each }k \ge \ol{K}.\]
         Putting these arguments together, \eqref{lin_quad_rate} can be written as
         \begin{align*}
             &\optgaperr_{k+1} \leqslant \Bigg( 1 - \eta \bigg(1 - \frac{\gamma}{\mu} \bigg) + \frac{\eta \ol{L}}{\mu \sqrt{N}} \constt_1 e^{-
             \rrate_1 k}\Bigg)\optgaperr_k \\ &
             \hspace{3cm} + \Bigg(\frac{\eta\overline{L}}{2 \mu} + \frac{\eta \ol{\constt}}{\mu \sqrt{N}} \Bigg)\optgaperr_k^2 \text{ for each }k\ge \ol{K}.
         \end{align*}

         \noindent \textbf{Step 2:} We now appeal to Theorem \ref{lin_convergence}, which states that for a sufficiently small stepsize \(\stsz\), the sequence \(\big(\optgaperr_k \big)_k\) is convergent sequence with limit point \(\ell = 0\). Therefore, we conclude that for every \(\epsilon>0\), there exists a \(K_{\epsilon} \in \Nz\) such that  whenever \(k > K_{\epsilon}> \ol{K}\), \(\optgaperr_k < \frac{ 2\mu \sqrt{N}}{\eta \big( \ol{L}\sqrt{N} + 2 \ol{\constt}\big)}\epsilon .\) 
         This immediately implies that for every \(k > K_{\epsilon}\)
        \begin{align*}
            &1 - \eta \Big(1 - \frac{\gamma}{\mu} \Big) + \frac{\eta \ol{L}}{\mu \sqrt{N}} \constt_1 e^{- \rrate_1 k} +\bigg( \frac{\eta\overline{L}}{2 \mu} + \frac{\eta \ol{\constt}}{\mu \sqrt{N}} \bigg)\optgaperr_k \\
            & < 1 - \eta \Big(1 - \frac{\gamma}{\mu} \Big) + \underbrace{\frac{\eta \ol{L}}{\mu \sqrt{N}} \constt_1 e^{- \rrate_1 k} + \epsilon}_{\teL \ol{\epsilon}_k}, 
        \end{align*}
        where the sequence converges \((\ol{\epsilon_k})_{k \in \Nz} \xrightarrow[k \ra +\infty]{}\epsilon\). For ease of notations, we define \(\xi_k \Let 1 - \eta \big(1 - \frac{\gamma}{\mu} \big) + \ol{\epsilon}_k\) and  \(\xi_k \xrightarrow[k \ra +\infty]{} \xi \Let 1 - \eta \big(1 - \frac{\gamma}{\mu} \big) + \epsilon\).
        Therefore, for each \(k \in \Nz\), \(\optgaperr_{k+1} \leqslant \xi_k \optgaperr_k \leqslant \Big(\prod_{k \in \Nz} \xi_k \Big) \optgaperr_0\). Observe that \(\log \xi_k \xrightarrow[k \ra + \infty]{} \log \xi \) and \(\frac{1}{k+1} \log \optgaperr_0 \xrightarrow[ k \ra +\infty]{}0\), which implies that \(\frac{1}{k+1} \sum_{k \in \Nz} \log \xi_k \xrightarrow[k \ra +\infty]{}\log \xi\) for every \(\epsilon>0\).  To see why this is true, we compute the upper bound 
        \begin{align*}
         \norm{\frac{1}{k+1} \sum_{k \in \Nz} \log \xi_k - \log \xi} 
         &\le \frac{1}{k+1} \sum_{k \in \Nz} \abs{\log \xi_k - \log \xi} \\&\xrightarrow[k \ra +\infty]{} 0,
        \end{align*}
        which allows us to write
        \begin{align*}
            \liminf_{k \ra +\infty}\bigg(- & \frac{1}{k+1} \log \optgaperr_{k+1}\bigg) \ge \liminf_{k \ra +\infty} \bigg(- \frac{1}{k+1} \sum_{k \in \Nz} \log \xi_k \bigg)\\
            & \xrightarrow[k \ra +\infty]{} - \log \bigg(1 - \eta \Big(1 - \frac{\gamma}{\mu} \Big) + \epsilon \bigg)
        \end{align*}
        for every sufficiently small \(\epsilon>0\). 
        
        As an immediate consequence of \eqref{eq:geo_rate}, we have
        \begin{align*}
            -\limsup_{k \ra + \infty} \bigg(\frac{\log \optgaperr_k}{k+1} \bigg) &= \liminf_{k \ra + \infty} \bigg(-\frac{\log \optgaperr_k}{k+1} \bigg) \optgaperr_{k+1}\bigg) \\
            & \ge - \log \bigg(1 - \eta \Big(1 - \frac{\gamma}{\mu} \Big) \bigg),
        \end{align*}
        which implies that \[\limsup_{k \ra + \infty} \bigg(\frac{\log \optgaperr_k}{k+1} \bigg) \le \log \bigg(1 - \eta \Big(1 - \frac{\gamma}{\mu} \Big) \bigg). \]
        
        This completes the proof of Theorem \ref{th:linear-quadratic rate}.
    
\end{appendices}    

	\bibliographystyle{IEEEtran}
	\bibliography{IEEEabrv,utils/refs1}

\vspace*{-1.25cm}
    \begin{IEEEbiography}[{\includegraphics[width=1in,height=1.25in,clip,keepaspectratio]{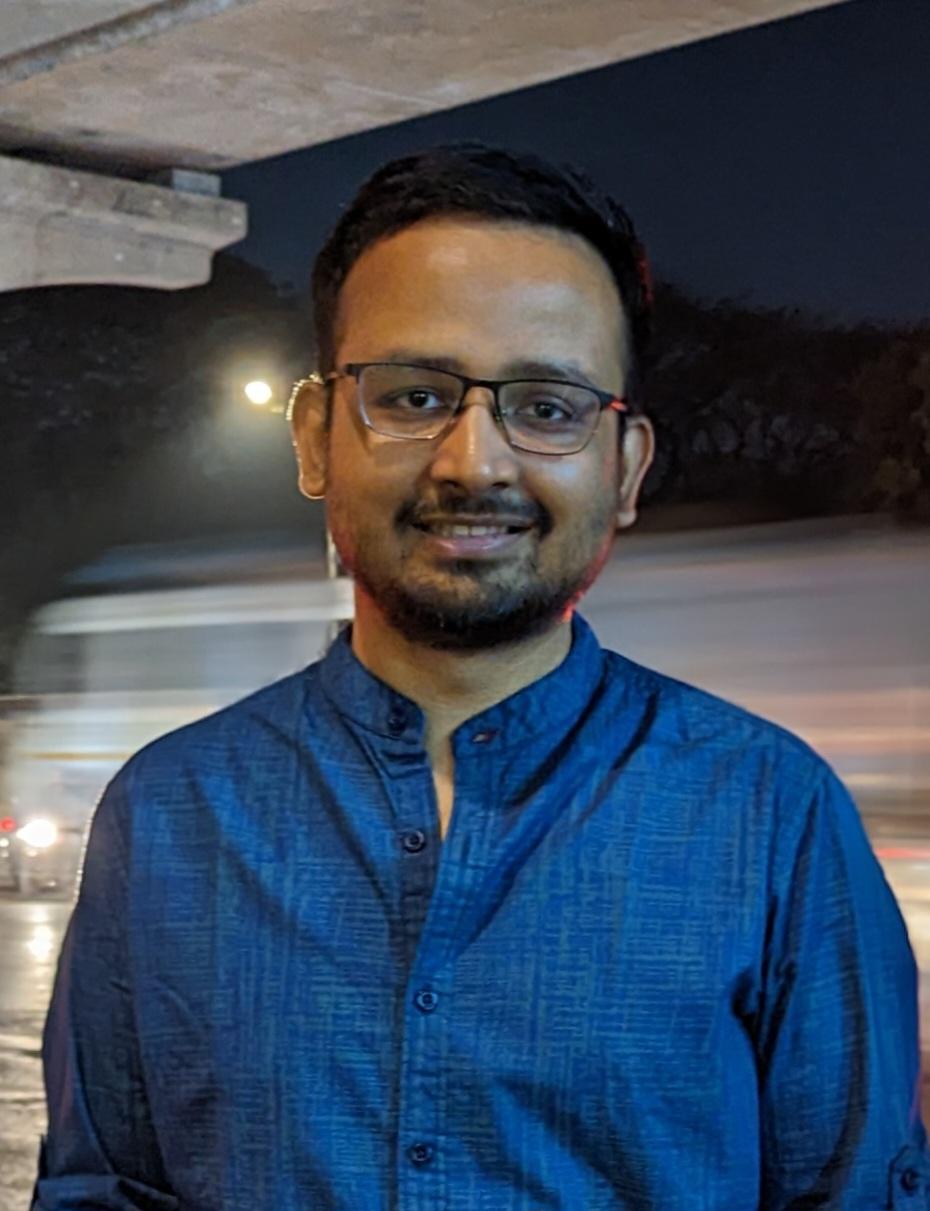}}]{Souvik Das} is a postdoctoral researcher in the Department of Electrical Engineering at Uppsala University, Sweden. He received his doctoral degree from the Centre for Systems and Control at IIT Bombay in 2025. His research interests are broadly in the field of optimization and optimization-based control, and security and privacy of cyber-physical systems.
\end{IEEEbiography}

\vspace*{-1.5cm}
\begin{IEEEbiography}[{\includegraphics[width=1in,height=1.25in,clip,keepaspectratio]{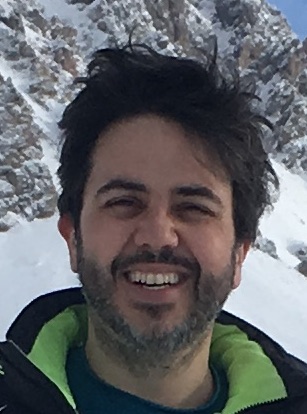}}]{Luca Schenato} (Fellow, IEEE) received the
Dr.Eng. degree in electrical engineering from
the University of Padova, Padova, Italy, in 1999,
and the Ph.D. degree in electrical engineering
and computer sciences from the University of
California (UC), Berkeley, Berkeley, CA, USA, in
2003. He was a Postdoctoral Researcher in 2004
and Visiting Professor during 2013–2014 with
UC Berkeley. He is currently a Full Professor
with the Information Engineering Department,
University of Padova. His interests include networked control systems, multi-agent systems, wireless sensor networks, distributed optimisation, and synthetic biology. Luca Schenato has been awarded the 2004 Researchers Mobility Fellowship by the Italian Ministry of Education, University and Research (MIUR), the 2006 Eli Jury Award in U.C. Berkeley and the EUCA European Control Award in 2014, and IEEE Fellow in 2017. He served as Associate Editor for IEEE Trans. on Automatic Control from 2010 to 2014, and he is currently Senior Editor for IEEE Trans. on Control of Network Systems and Associate Editor for Automatica.
\end{IEEEbiography}
\vspace*{-1cm}
 \begin{IEEEbiography}[{\includegraphics[width=1in,height=1.25in,clip,keepaspectratio]{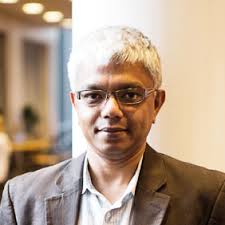}}]{Subhrakanti Dey} (Fellow, IEEE) received the Ph.D. degree from the Department of Systems Engineering, Research School of Information Sciences and Engineering, Australian National University, Canberra, in 1996. 
He is currently a Professor and Head of the Signals and Systems division in the Dept of Electrical Engineering at Uppsala University, Sweden. He has also held professorial positions at NUI Maynooth, Ireland, and the University of Melbourne, Australia. His current research interests include networked control systems, distributed machine learning and optimization, and detection and estimation theory for wireless sensor networks. He is currently a Senior Editor for IEEE Transactions of Control of Network Systems and IEEE Transactions on Signal and Information Processing over Networks, and an Associate Editor for Automatica. He also served as a Senior Editor for IEEE Control Systems Letters until 2025.

\end{IEEEbiography}
\end{document}